\documentclass[11pt,a4paper, reqno]{amsart}

\usepackage{amssymb,amsfonts}
\usepackage[widespace]{fourier}
\usepackage[all,cmtip]{xy}

\usepackage{latexsym}
\usepackage{color}
\usepackage{hyperref}
\hypersetup{colorlinks=true,linkcolor=blue, citecolor=blue}

\numberwithin{equation}{section}
\numberwithin{equation}{subsection}
\newcommand{\Z}{\mathbb{Z}}
\newcommand{\C}{\mathbb{C}}
\newcommand{\R}{\mathbb{R}}
\newcommand{\Q}{\mathbb{Q}}
\newcommand{\bp}{{\mathbb P}}

\newtheorem{theorem}{Theorem}[section]
\newtheorem{proposition}[theorem]{Proposition}
\newtheorem{lemma}[theorem]{Lemma}
\newtheorem{corollary}[theorem]{Corollary}
\newtheorem{remark}{\bf Remark}[section]

\newtheorem{fact}{Fact}[section]
\newtheorem{definition}{Definition}

\begin{document}
\title[On a Relative Mumford-Newstead Theorem]{On a Relative Mumford-Newstead Theorem}
\author{Suratno Basu}
\email{suratno@cmi.ac.in}
\address{Chennai Mathematical Institute, H1 SIPCOT IT Park, Siruseri, Kelambakkam-603103, India.}
\begin{abstract}
 In this paper, we prove a relative version of the classical Mumford-Newstead theorem
for a family of smooth curves degenerating to a reducible curve with a simple node. We also prove
a Torelli-type theorem by showing that certain moduli spaces of torsion free sheaves on a reducible
curve allows us to recover the curve from the moduli space.
\end{abstract}
\maketitle

 \section{Introduction}
  Let  $X$ be a smooth, projective curve of genus $g\geq 2$ over $\C$. 
 We fix a line bundle $L$ of odd degree  over $X$. Let $M_{_{X}}$ be the moduli space of rank 2, 
 stable vector bundles $E$ such that det$E\simeq L$. It is known that $M_{_X}$ is a smooth, 
 projective and unirational variety. Consequently it follows, by \cite[Lemma 1]{Se}, that the
 Hodge numbers $h^{0,p}=h^{p,0}=0$.
  Therefore, we have the following Hodge decomposition:
 \[H^3(M_{_{X}},\C)=H^{1,2}\oplus \overline{H^{1,2}},\] where $\overline{\alpha}$ is the complex conjugate of $\alpha \in H^3(M_{_X},\C)=H^3(M_{_X},\R)\otimes \C$ and $H^{1,2}\simeq H^2(M_{_X},\Omega^1_{_{M_{_X}}})$. 
 Let $pr_1:H^3(M_{_{X}},\C)\rightarrow H^{1,2}$ be the first projection.  
 Since $H^3(M_{_X},\R)\cap H^{1,2}=\{0\}$, we get that the image $pr_1(H^3(M_{_{X}},\Z))$ is a full lattice in $H^{1,2}$. 
  We associate a complex torus corresponding to the above Hodge structure: \begin{equation}\label{inter}J^2(M_{_X}):=\frac{H^{1,2}}{pr_1(H^3(M_{_{X}},\Z))}
  \end{equation} It is known as the second  intermediate Jacobian of $M_{_{X}}$. 
  We remark that the complex torus, defined above, varies holomorphically in an analytic family of smooth projective, unirational varieties and is a principally polarised abelian variety. It is known that the second Betti number $b_{_2}(M_{_X})=1$ (\cite{ne}). Let $\omega$ be the unique ample, integral, K\"ahler class on $M_{_X}$. Then the principal polarisation on $J^2(M_{_{X}})$ is induced by the following paring:
  \begin{equation}(\alpha,\beta) \mapsto  \int_{_{M_X}} \omega^{n-3}\wedge \alpha \wedge \overline{\beta},\end{equation} where $\alpha,\beta \in H^{1,2}$ and $n=dim_{_{\C}}M_{_{X}}$. 
  We denote this polarisation on $J^2(M_{_X})$ by $\theta'$. 
  The theorem of  Mumford and Newstead  (\cite[Theorem in page 1201]{mum-new}) asserts that there is a natural isomorphism $\phi:J(X)\to J^2(M_{_X})$ such that $\phi^{*}(\theta')=\theta$,
  where $J(X)$ is the Jacobian of the curve and $\theta$ is the canonical polarisation on $J(X)$. In \cite[Section 5, page 625]{Ba2}) there is a detailed proof of the fact $\phi^{*}(\theta')=\theta$. Hence, appealing to the classical Torelli theorem one can recover the curve $X$ from the moduli space $M_{_X}$. 
 
Let $X_{_0}$ be a projective curve with exactly two smooth irreducible components $X_{_1}$ and $X_{_2}$ meeting at a simple node $p$. Fix two rational numbers $0<a_{_1},a_{_2}<1$ such that $a_{_1}+a_{_2}=1$ and let $\chi$ be an odd integer. 
 Under some numerical conditions, Nagaraj and Seshadri construct in \cite[Theorem 4.1]{ns},
 the moduli space $M(2,(a_{_1},a_{_2}),\chi)$ of rank $2$, $(a_{_1},a_{_2})$-semistable torsion free sheaves on $X_{_0}$ with  Euler charachteristic $\chi$. Moreover, they show that $M(2,(a_{_1},a_{_2}),\chi)$ is the union of two smooth, projective varieties intersecting transversally  along a smooth divisor. We will observe that there exists a
 determinant morphism 
 $det:M(2,(a_{_1},a_{_2}),\chi)\to J^{\chi-(1-g)}(X_{_0})$ 
  where $J^{\chi-(1-g)}(X_{_0})$ is the Jacobian parametrising the line bundles with Euler characteristic $\chi-(1-g)$ over $X_{_0}$ (see Proposition \ref{determin} in Appendix). We further observe that the fibres of the morphism $det$ is again the union of two smooth projective varieties intersecting transversally (see Proposition \ref{tran} in Appendix). Fix $\xi\in J^{\chi-(1-g)}$. We denote the fibre $det^{-1}(\xi)$ by $M_{_{0,\xi}}$.  Since  $M_{_{0,\xi}}$ is a singular variety, a priori $H^3(M_{_{0,\xi}},\C)$ has an intrinsic mixed Hodge structure. Let $g$ be the arithmetic genus of $X_{_0}$. Note that $g=g_{_1}+g_{_2}$, where $g_{_i}$ is the genus of $X_{_i}$ for  $i=1,2$.
Under the assumption $g_{_i}>3$, $i=1,2$,  we will show that $H^3(M_{_{0,\xi}},\Q)\simeq \Q^{2g}$, and that it has a pure Hodge structure with  Hodge numbers $h^{^{3,0}}=h^{^{0,3}}=0$. Thus we have an intermediate Jacobian $J^2(M_{_{0,\xi}})$, as defined earlier, corresponding  to the Hodge structure on $H^3(M_{_{0,\xi}},\C)$ which is a priori only a complex torus of dimension $g$. 

 Let $\pi:\mathcal{X}\to C$ be a proper, flat and surjective family of curves, parametrised by a smooth, irreducible curve $C$. Fix $0\in C$. We assume that $\pi$ is smooth outside the point $0$ and $\pi^{-1}(0)=X_{_0}$, where $X_0$ is as above, $g_i>3$ for $i=1,2$. 
 Let $X_{_t}$ be the fibre $\pi^{-1}(t)$ over $t\in C$.
 Fix a line bundle $\mathcal{L}$ over $\mathcal{X}$ such that the restrictions $\mathcal{L}_{_t}$ to $X_{_t}$  are line bundles with Euler characteristics $\chi-(1-g)$ for $t\neq 0$ and $\mathcal{L}_{_0}$ is isomorphic to the line bundle $\xi$. 
 In this situation, it is observed in \cite[Lemma 7.2]{ns} that there is a family $\pi':\mathcal{M}_{_{\mathcal{L}}}\to C$ such that the fibre $\pi'^{-1}(t)$ over a point $t\neq 0$ is  $M_{_{t,\mathcal L_t}}$, the moduli space of rank $2$, semistable locally free sheaves with $det\simeq \mathcal L_t$ over the smooth projective curve $X_{_t}$ and $\pi'^{-1}(0)=M_{_{0,\xi}}$ (see Section \ref{degen}). We should mention a related work by X Sun \cite{sun}. In \cite{sun} the author constructs a family of rank $r$ fixed determinant, semistable bundles over smooth projective curves degenerating to a ``fixed determinant'' moduli space of rank $r$ torsion free sheaves over $X_{_0}$. Though his methods are different we believe, in rank $2$ case, the relative moduli space in \cite{sun} coincides with $\mathcal{M}_{_{\mathcal{L}}}$. 
 We consider an analytic disc $\Delta$ around the point $0$ and we denote the family $\pi':\pi'^{-1}(\Delta)\to \Delta$ by $\{M_{_{t, \mathcal L_t}}\}_{_{t\in \Delta}}$.

 With these notations we state one of the main results of this paper :
 \begin{theorem}\label{m}{\ }
 
 \begin{enumerate}
  \item There is a holomorphic family  $\{J^2(M_{t,\mathcal{L}_{t}})\}_{_{t\in \Delta}}$ of intermediate Jacobians corresponding to the family $\{M_{_{t, \mathcal L_t}}\}_{_{t\in \Delta}}$. In other words, there is a surjective,
   proper, holomorphic submersion 
    \[\pi_{_2} : J^2(\mathcal{M}_{_{\mathcal{L}}})\longrightarrow \Delta\] 
   such that $\pi_2^{-1}(t)=J^2(M_{_{t, \mathcal L_t}})~ \forall ~  t\in \Delta^{*}:=\Delta \setminus \{0\}$ and $\pi_2^{-1}(0)=J^2(M_{_{0,\xi}})$.   
   Further, there exists a relative ample class $\Theta'$ on $J^2(\mathcal{M}_{_{\mathcal{L}}})_{|\Delta^*}$ such that $\Theta'_{|J^2(M_{_{t,\mathcal L_t}})}=\theta'_{_t}$, where $\theta'_{_t}$ is the principal polarisation on $J^2(M_{_{t, \mathcal L_t}})$.
   
  \item 
   There is an isomorphism 
   \begin{equation}
\xymatrix{
J^0(\mathcal{X}) \ar[rr]^{\Phi}_{\sim} \ar[rd]_{\pi_{_1}} && J^2(\mathcal{M}_{_{\mathcal{L}}})\ar[ld]^{\pi_{_2}} \\
& \Delta
}
\end{equation} 
such that $\Phi^*\Theta'_{|\pi_1^{-1}(t)}=\theta_{_t}$ for all $t\in \Delta^*$, where $\pi_1: J^0(\mathcal X)\to \Delta$ is the holomorphic family $\{J^0(X_{_t})\}_{_{t\in \Delta}}$ of Jacobians and $\theta_{_t}$ is the canonical polarisation on $J^0(X_{_t})$. In particular $J^2(\mathcal{M}_{_{\mathcal{L}}})_{_0}:=\pi_{_2}^{-1}(0)$ is an abelian variety.
\end{enumerate}
\end{theorem}
 
  %We note that the fibre $\pi_{_1}^{-1}(0)$ is isomorphic to the Jacobian $J^0(X_{_0})$ of $X_{_0}$

  By the above theorem we deduce the following:
  
    \begin{corollary} 
  Let $X_{_0}$ be a  projective curve with exactly two smooth irreducible components $X_1$ and $X_2$ meeting at a simple node $p$. We further assume that $g_{_i}>3$, $i=1,2$. Then,  
  there is an isomorphism $J^0(X_{_0})\simeq J^2(M_{_{0,\xi}})$, where $\xi\in J^{\chi}(X_{_0})$.
   In particular, $J^2(M_{_{0,\xi}})$ is an abelian variety.
  \end{corollary}
  
  Since $J^0(X_{_0})$ is isomorphic to $J^0(X_{_1})\times J^0(X_{_2})$, 
  we observe the Jacobian  $J^0(X_{_0})$ is independent of the nodal point in $X_{_0}$.
  Hence, the classical Torelli theorem fails for such curves
  (see \cite[Page 6 ]{mummay}).
  On the other hand, it is known that under suitable choice of the polarisation on the Jacobian $J^0(X_{_0})$, one can recover the normalization $\tilde {X_{_0}}$ of $X_{_0}$, but not the curve $X_{_0}$. In otherwords one can recover both the components of $X_{_0}$ but not the nodal point(see \cite[page 125]{RH}). 
  
    We see that the moduli space $M_{_{0,\xi}}$ of rank $2$ torsion free sheaves carries more information than the Jacobian $J(X_{_0})$. In fact, we show that we can actually recover the curve $X_{_0}$ from $M_{_{0,\xi'}}$ by following a strategy of \cite{bbd}.  
More precisely, we will prove the following analogue of the Torelli theorem for reducible curves:
\begin{theorem}
  Let $X_{_0}$ ( resp. $Y_{_0}$) be the projective curve with two smooth irreducible components $X_i$ (resp. $Y_i$), $i=1,2$ meeting at a simple node $p$ (resp. $q$).
  We assume that $\mbox{genus}(X_{_i})=\mbox{genus}(Y_{_i})$, for $i=1,2$, and $X_{_1}\ncong X_{_2}$(resp. $Y_{_1}\ncong Y_{_2}$). 
  Let  $M_{_{0,\xi_{_{X_0}}}}$ (resp. $M_{_{0,\xi_{_{Y_0}}}}$) be the moduli space of rank $2$, semistable torsion free sheaves  $E$ with $detE\simeq \xi_{_{X_0}}$, $\xi_{_{X_0}}\in J^{\chi}(X_{_0})$, on $X_{_0}$ (resp. on $Y_{_0}$). 
  If $M_{_{0,\xi_{_{X_0}}}}\simeq M_{_{0,\xi_{_{Y_0}}}}$ then we have $X_{_0}\simeq Y_{_0}$.
  \end{theorem}
  {\it Acknowledgements: I am extremely thankful to Prof V. Balaji who introduced me to this problem and discussed this work with me. I thank Prof. C.S Seshadri for suggesting a way to define a certain ''determinant'' morphism and for several helpful discussions. I thank B Narasimha Chary for a very careful reading of the previous draft and suggesting many changes. I have greatly benefited from the discussions with Dr. Ronnie Sebastian. I also thank Prof. D.S Nagaraj, Rohith Varma and Krisanu Dan for many helpful discussions. Finally I wish to thank the referee for being extremely patient with the previous manuscript and generously suggesting many changes. The proofs of the Proposition \ref{determin} and \ref{reldet} are suggested by the referee.  }
  
 \section{Preliminaries}
In this section, we briefly recall the main results in \cite{ns}
which will be extensively used in the present work. 
Before proceeding further we will fix the following notations:
\subsubsection{Notation}\label{note1}
\begin{itemize}
 \item Throughout we work over the field $\mathbb C$ of complex numbers. We assume that all the schemes are reduced, separated and finite type over $\C$.
 \item Let $p_{_i}: X_{_1}\times \cdots \times X_{_n}\to X_{i}$ be the $i^{th}$ projection, where $X_{_i}$ is a scheme for $i=1,\ldots,n$. 
 By abuse of notation, we denote $p_{_i}^{*}(E_{_i})$ also by $E_{_i}$, where $E_{_i}$ is a sheaf of $\mathcal{O}_{_{X_{_i}}}$ module.
 \item Let $X$ be a projective scheme and $E$ be a vector bundle over $X$. Then we set $h^i(E):=dim_{_{\C}}H^i(X,E)$. Let $S$ be another projective scheme and $\mathcal{E}$ be a coherent sheaf over $X\times S$ then we set $\mathcal{E}_{_s}:=\mathcal{E}|_{_{X\times s}}$, $s\in S$.
 \item By cohomology of a scheme $X$, we mean the singular cohomology of the space $X_{_{ann}}$, the analytic space with complex analytic topology associated to $X$.
 \item Let $E$ be coherent sheaf over $X$. We denote by $E(p):=\frac{E_{_p}}{m_{_p}E_{_p}}$ the fibre of E at $p\in X$. 
 \item
 Let $X$ be a smooth projective curve and $E$ be a vector bundle over $X$.  Then we denote $E\otimes \mathcal{O}_{_X}(np)$ by $E(np)$, where $p\in X$ is a closed point and $n$ is an integer.
 \item If $Z$ is a closed subvariety of a smooth variety $X$, then we denote by $Codim(Z,X)$, the codimension of $Z$ in $X$.
\end{itemize}

\subsection{Triples associated to a torsion free sheaves on a reducible nodal curve}
%\subsection{Moduli space of rank \texorpdfstring{$2$}{} torsion free sheaves over a reducible nodal curve}
Let $X_{_0}$ be a projective curve of arithmetic genus $g$ with exactly two smooth irreducible components $X_{_1}$ and $X_{_2}$ meeting at a simple node $p$. The arithmetic genus $g$ of such a curve is 
$g=g_1+g_2$, where $g_{_i}$ is the genus of $X_{_i}$ for $i=1,2$. 

By a torsion free sheaf over $X_0$ we always mean a coherent $\mathcal{O}_{_{X_{0}}}$-module of depth $1$. Let $\stackrel{\to}{C}$  be a category whose objects are triples $(F_{_1},F_{_2},A)$ where $F_{_i}$ are vector bundles on $X_{_i}$, for $i=1,2$ and $A:F_{_1}(p)\to F_{_2}(p)$ is a linear map. %Let $F_{_1}$ (resp. $F_{_2}$) be a vector bundle over $X_{_1}$ (resp. over $X_{_2}$) and ${A}:F_{_1}(p)\to F_{_2}(p)$ be a linear homomorphism. Then we call $(F_{_1},F_{_2},A)$. We define a category  $\stackrel{\to}{C}$  whose objects are  and a morphism between such triples is defined to be: 
Let $(F_{_1},F_{_2},A), (G_{_1},G_{_2},B)\in \stackrel{\rightarrow}{C}$. We say $\phi: (F_{_1},F_{_2},A)\to (G_{_1},G_{_2},B)$ is a morphism if there are morphisms $\phi_{_i}: F_{_i}\to G_{_i}$  of $\mathcal{O}_{_{X_{_i}}}$-modules for $i=1,2$ such that the following diagram is commutative:
 \begin{equation}\label{mor1}
                     \xymatrix{F_{_1}(p) \ar[r]^{\phi_1 \otimes k(p) } \ar[d]^{A} & G_{_1}(p) \ar[d]^{B} \\
                     F_{_2}(p) \ar[r]^{\phi_2 \otimes k(p)} & G_{_2}(p)}
                    \end{equation}

In \cite[Lemma 2.8]{ns}, it is shown that there is an equivalence of categories between $\stackrel{\to}{C}$ and the category of torsion free sheaves over $X_{_0}$.

\begin{remark}\label{direction}
Similarly, we define another category $\stackrel{\leftarrow}{C}$ whose objects are triples $(F_{_1},F_{_2},A)$ where $F_{_i}$ are vector bundles over $X_{_i}$ for $i=1,2$ and $A:F_{_2}(p)\to F_{_1}(p)$ is a linear map. The morphism between any two such triples is defined in the same way before. The category of torsion free sheaves is equivalent to the category $\stackrel{\leftarrow}{C}$ (see \cite[Remark 2.9]{ns}). Now if the triples $(F_{_1},F_{_2},A)\in \stackrel{\to}{C}$ and $(F'_{_1},F'_{_2},B)\in \stackrel{\leftarrow}{C}$ correspond the same torsion free sheaf $F$, then they are related by the following diagram:
\begin{equation}\label{mor2}
                     \xymatrix{F_{_1}(p) \ar[r]^{i_{_p}} \ar[d]^{A} & F'_{_1}(p)  \\
                     F_{_2}(p)  &\ar[l]^{j_{_p}} \ar[u]^{B} F'_{_2}(p)}
                    \end{equation}
                    where $i:F_{_1}\to F'_{_1}$ (resp. $j:F'_{_2}\to F_{_2}$) is a morphism of vector bundle which is an isomorphism outside the point $p$ and $ker(i_{_p})=ker(A)$ (resp $Im(j_{_p})=Im(A)$)(see \cite[Remark 2.5]{ns}).  $F_{_i}'$ is called the Hecke-modification of $F_{_i}$ for $i=1,2$.
 \end{remark}                   
\subsubsection{Notion of semistability}
Fix an ample line bundle $\mathcal{O}_{_{X_0}}(1)$ on $X_{_0}$. Let $deg(\mathcal{O}_{_{X_{_0}}}(1)|_{_{X_i}})=c_{_i}$, $i=1,2$, and $a_{_i}=\frac{c_{_i}}{c_{_1}+c_{_2}}$. Then $0<a_{_1},a_{_2}<1$ and $a_{_1}+a_{_2}=1$. We say $a=(a_{_1},a_{_2})$ a polarisation on $X_{_0}$. A torsion free sheaf $F$ on $X_{_0}$ is of rank type $(r_{_1},r_{_2})$ if the generic rank of the restrictions $F|_{_{X_{_i}}}$ are $r_{_i}$, $i=1,2$.
\begin{definition}\label{semi}
  For a torsion free sheaf $F$ of rank type $(r_{_1},r_{_2})$, we define the rank $r:=a_{_1}r_{_1}+a_{_2}r_{_2}$ and the slope $\mu(F):=\frac{\chi(F)}{r}$, where $\chi(F):=h^0(F)-h^1(F)$. A torsion free sheaf $F$ is said to be semistable(resp. stable) with respect to the polarisation $a=(a_{_1},a_{_2})$ if $\mu(G)\leq \mu(F)$(resp. $<$) for all nontrivial proper subsheaves $G$ of $F$. We define the Euler characteristic and the slope  of a triple $(F_{_1},F_{_2},A)\in \stackrel{\to}{C}$ to be: 
  \begin{equation}
  \chi((F_{_1},F_{_2},A))=\chi(F_1)+\chi(F_2)-rk(F_2)~ \mbox{and}~ \mu(F_1,F_2,A)=\frac{\chi((F_1,F_2,A))}{r}.
  \end{equation}
  A triple  $(F_1,F_2,A)$ is said to be semistable(resp. stable) if $\mu(G_1,G_2,B)\leq \mu(F_1,F_2,A)$ for all nontrivial proper subtriples of $(F_1,F_2,A)$ (for definition of a subtriple see \cite[Definition 2.3]{ns}).  
  \end{definition}

  \begin{remark}
If a torsion free sheaf $F$ is associated to a triple $(F_1,F_2,A)$ then $\chi(F)=\chi(F_1,F_2,A)$(see \cite[Remark 2.11]{ns}). We have already remarked the category of torsion free sheaves is equivalent to the category of triples in a fixed direction. Therefore, a torsion free sheaf $F$ is $a=(a_{_1},a_{_2})$-semistable (resp. stable) if and only if the corresponding triple $(F_1,F_2,A)$ is  $a=(a_{_1},a_{_2})$-semistable (resp. stable).
\end{remark}
\subsection{Moduli space of rank \texorpdfstring{$2$}{} torsion free sheaves over a reducible nodal curve}
\subsubsection{Euler Characteristic bounds for rank $2$ semistable sheaves}
Fix an integer $\chi$ and a polarisation $a=(a_{_1},a_{_2})$ on $X_{_0}$ such that $a_{_1}\chi$ is not an integer. Then we have the following Euler characteristic restrictions: 
\begin{lemma}\label{bounds2}
 Let $\chi_{_1}$, $\chi_{_2}$ be the unique integers satisfying 
\begin{equation}\label{inq1}
a_{_1}\chi< \chi_{_1}< a_{_1}\chi+1 ~ , ~  a_{_2}\chi+1< \chi_{_2}< a_{_2}\chi+2
\end{equation} 
and $\chi=\chi_{_1}+\chi_{_2}-2$. If $F$ is a rank $2$, $a=(a_{_1},a_{_2})$-semistable sheaf then $\chi(F_{_1})=\chi_{_1}$, $\chi(F_{_2})=\chi_{_2}$ or $\chi(F_{_1})=\chi_{_1}+1$, $\chi(F_{_2})=\chi_{_2}-1$ and $rk(A)\geq1$ where $(F_{_1},F_{_2},A)\in \stackrel{\to}{C}$ is the unique  triple representing $F$. Moreover if $F$ is non-locally free then $\chi(F_{_1})=\chi_{_1}$ and $\chi(F_{_2})=\chi_{_2}$.
\begin{proof}
 See \cite[Theorem 3.1]{ns}.
\end{proof}
\end{lemma}
\textbf{For the rest of the paper we fix an odd integer $\chi$ and a polarization $a:=(a_{_1},a_{_2})$ ($a_{_1}<a_{_2}$) on $X_{_0}$ such that $a_{_1}\chi$ is not an integer}. 
With these notations, one of the main results of \cite{ns} is the following:
\begin{theorem}(\cite[Theorem 4.1]{ns})
  The moduli space $M(2,a,\chi)$ of isomorphism classes rank $2$, $(a_{_1},a_{_2})$ stable torsion free sheaves exists as a reduced,connected, projective scheme. Moreover, it has two smooth, irreducible components meeting transversally along a smooth divisor $D$.
  \end{theorem}
  \subsection{Fixed determinant moduli space}
  Let $J^{\chi_{_i}-(1-g_{_i})}(X_{_i})$ be the Jacobian of isomorphism classes of line bundles over $X_{_i}$ with Euler characteristic $\chi_{_i}-(1-g_{_i})$, $i=1,2$ and $J_{_0}:=J^{\chi_{_1}-(1-g_{_1})}(X_{_1})\times J^{\chi_{_2}-(1-g_{_2})}(X_{_2})$. In the Appendix we will show that there is a well defined determinant morphism $det:M(2,a,\chi)\to J_{_0}$ whose fibres are again the union of two sooth, projective varieties intersecting transversally along a smooth divisor ( see Proposition \ref{determin}). %In the next section we will compute the cohomologies of a fibre of the $det$ morphism. Later in this section we will give a alternative construction of this space.

\subsection{Moduli space of triples}
Fix $\xi\in J_{0}$ and let $det^{-1}(\xi):=M_{_{0,\xi}}$. In this subsection we will discuss a different description of the moduli spaces $M(2,a,\chi)$ and $M_{_{0,\xi}}$ in terms of certain moduli space of triples glued along a certain divisor. This description is given in section $5$ of the article \cite{ns}. This description will be useful for the cohomology computations later. 

The following facts are well known. For the completeness we shall indicate a proof.
\begin{fact}\label{para} 
  Let $(X,x)$ be a smooth,projective curve together with a marked point $x$ and $(E,0\subset F^2E(x)\subset E(x))$ be a parabolic vector bundle with weights $0<\beta_{_1}<\beta_{_2}< 1$. Suppose the weights satisfy $|\beta_{_1}-\beta_{_2}|<\frac{1}{2}$. Then we have-\\
  $(a)$ $E$ is parabolic semistable implies $E$ is parabolic stable.\\
  $(b)$ $E$ is parabolic semistable implies $E$ is semistable.\\
  $(c)$ If $E$ is stable then any quasi parabolic structure $(E,0\subset F^2E(x)\subset E(x))$ is parabolic semistable with respect to the weights $0<\beta_{_1}<\beta_{_2}<1$.
  \begin{proof}
   From our assumption on weights we get that $|\frac {\beta_{_1}+\beta_{_2}}{2}-\beta_{_i}|<\frac {1}{2}$ for $i=1,2$. Suppose $E$ is strictly parabolic semistable. Let $L$ be a parabolic line subbundle of $E$. Then we have-\[deg(L)=\frac{deg(E)}{2}+\frac{\beta_{_1}+\beta_{_2}}{2}-\beta_{_i}.\] Since $|\frac {\beta_{_1}+\beta_{_2}}{2}-\beta_{_i}|<\frac {1}{2}$ and $deg(L)$ is an integer this is not possible. This completes the proof of $(a)$. Let $L$ be a line subbundle of $E$. The parabolic stability of $E$  implies \[deg(L)<\frac{deg(E)}{2}+\frac{\beta_{_1}+\beta_{_2}}{2}-\beta_{_i}.\] Therefore, $deg(L)<\frac{deg(E)}{2}\pm \frac{1}{2}$. Since $deg(L)$ is an integer the above inequality will imply $deg(L)\leq \frac{deg(E)}{2}$. This completes the proof of $(b)$. Let $L$ be a subbundle of $E$. If $L(x)\cap F^2E(x)\neq 0$ then we associate the weight $\beta_{_1}$ otherwise we associate the weight $\beta_{_2}$. Now as $E$ is stable we have \[deg(L)<\frac{deg(E)}{2}.\] Since $|\frac {\beta_{_1}+\beta_{_2}}{2}-\beta_{_i}|<\frac {1}{2}$ and $deg(L)$ is an integer we conclude that \[deg(L)<\frac{deg(E)}{2}+\frac{\beta_{_1}+\beta_{_2}}{2}-\beta_{_i}.\] This completes the proof of $(c)$.
  \end{proof}
\end{fact}
 
 The following result is proved in \cite{ns}
 \begin{fact}\label{semistable}
 Let $(F_{_1},F_{_2},A)\in \stackrel{\to}{C}$ (resp. $(F_{_1}',F_{_2}',B)\in \stackrel{\leftarrow}{C}$) be a rank $2$, $(a_{_1},a_{_2})$-semistable 
 and the Euler characteristic $\chi(F_{_i})$, $i=1,2$, satisfy the inequality \ref{inq1}(resp. the inequality \ref{inq3}, 
 then $F_{_i}$(resp. $F_{_i}'$) are semistable over $X_{_i}$ for $i=1,2$ (see \cite[Theorem 5.1]{ns}).
 \end{fact}

 Conversely we have the following:
 \begin{lemma}\label{semistable1}

 Let $F_{_i}$ be rank $2$ semistable bundles over $X_{_i}$ and the Euler characteristic $\chi(F_{_i})$, $i=1,2$, satisfy the inequalities  \ref{inq1}. Let $A:F_{_1}(p)\to F_{_2}(p)$ be a linear map and $rk(A)=2$, then $(F_{_1},F_{_2},A)\in \stackrel{\to}{C}$ is $(a_{_1},a_{_2})$-semistable. Moreover, if $F_{_1}$ and $F_{_2}$ are both stable then $(F_{_1},F_{_2},A)$ is $(a_{_1},a_{_2})$-semistable if $rk(A)\geq 1$.
  %\item Let $F_{_i}$ be rank $2$ stable bundles over $X_{_i}$, $i=1,2$ and the Euler characteristic $\chi(F_{_i})$, $i=1,2$ satisfies the inequality \ref{inq1}. Let $A:F_{_1}(p)\to F_{_2}(p)$ be a linear map and $rk(A)=1$, then then $(F_{_1},F_{_2},A)\in \stackrel{\to}{C}$ is $(a_{_1},a_{_2})$-semistable.
  \begin{proof}
  
  Case 1: Let $rk(A)=2$
   The proof of the statement $(1)$ follows from \cite[Lemma 3.1.12 page 39]{barik}.   
Now suppose $rk(A)=1$. In this case we need both $F_{_i}$ to be stable.

Since $rk(A)=1$ we get a parabolic structure on $F_{_1}$ given by $0\subset ker(A)\subset F_{_1}(p)$ and a parabolic structure on $F_{_2}(p)$ given by $0\subset Im(A)\subset F_{_2}(p)$. By Fact \ref{para} $(c)$ we conclude that the above two quasi parabolic structure are parabolic stable with respect to the weights $0<\frac{a_{_1}}{2}<\frac{a_{_2}}{2}<1$. Thus by \cite[Theorem 6.1]{ns} we get that $(F_{_1},F_{_2},A)$ is semistable.
\end{proof}
\end{lemma}
 \begin{remark}\label{other}
  The same results hold true for the triples in the other direction i.e if $F_{_i}$ are semistable over $X_{_i}$, $i=1,2$ satisfying the inequality \ref{inq3} and $rk(A)=2$ then the triple $(F_{_1},F_{_2},A)\in \stackrel{\leftarrow}{C}$ is $(a_{_1},a_{_2})$-semistable. Moreover, if $F_{_i}$ are stable and $rk(A)\geq 1$ then $(F_{_1},F_{_2},A)\in \stackrel{\leftarrow}{C}$ is $(a_{_1},a_{_2})$-semistable.
 \end{remark}

$(I)$ {\bf Semistable triple of type (I)}: We say a rank $2$, $(a_{_1},a_{_2})$-semistable triple $(F_{_1},F_{_2},A)\in \stackrel{\to}{C}$ is of type $(I)$ if $\chi(F_{_i})$, $i=1,2$, satisfy the following inequalities:
%Let $M_{_{12}}$ be the moduli space of rank $2$, $(a_{_1},a_{_2})$-semistable triple   $(F_{_1},F_{_2},A)\in \stackrel{\to}{C}$ satisfying: 
\begin{equation}\label{inq2} 
a_{_1}\chi\textrm{<} \chi_{_{X_{_1}}}(F_{_1})\textrm{<} a_{_1}\chi+1 ~ , ~  a_{_2}\chi+1\textrm{<} \chi_{_{X_{_2}}}(F_{_2})\textrm{<} a_{_2}\chi+2
\end{equation} 
and $rk(A)\geq 1$. 

$(II)$ {\bf Semistable triple of type (II)}: We say a $(a_{_1},a_{_2})$-semistable triple $(F_{_1},F_{_1},B)\in \stackrel{\leftarrow}{C}$ is of type $(II)$ if $\chi(F_{_i})$, $i=1,2$ satisfy the following inequalities:
\begin{equation}\label{inq3} 
a_{_1}\chi+1\textrm{<} \chi_{_{X_{_1}}}(F'_{_1})\textrm{<} a_{_1}\chi+2 ~ , ~  a_{_2}\chi\textrm{<} \chi_{_{X_{_2}}}(F'_{_2})\textrm{<} a_{_2}\chi+1\end{equation} and $rk(B)\geq 1$. 

Let $S$ be a scheme. We say $(\mathcal{F}_{_1},\mathcal{F}_{_2},\mathcal{A})$ a family of triples parametrised by $S$ if $\mathcal{F}_{_i}$'s are locally free sheaves on $X_{_i}\times S$, $i=1,2$ and $\mathcal{A}:\mathcal{F}_{_1}|_{_{p\times S}}\to \mathcal{F}_{_2}|_{_{p\times S}}$ is a $\mathcal{O}_{_S}$-module homomorphism of locally free sheaves.
\begin{remark}\label{family}
 Given a family of triples $(\mathcal{F}_{_1},\mathcal{F}_{_2},\mathcal{A})$ parametrised by $S$ we can associate a family of torsion free sheaves $\mathcal{F}$ parametrised by $S$ i.e a coherent sheaf $\mathcal{F}$ on $X_{_0}\times S$ which is flat over $S$ such that $\mathcal{F}_{_s}$ is torsion free for all $s\in S$. The association is the following: Let $G$ be the locally free subsheaf of $\mathcal{F}_{_1}|_{_{p\times S}}\oplus \mathcal{F}_{_2}|_{_{p\times S}}$ generated by the graph of the homomorphism $\mathcal{A}$ and $\mathcal{L}_{_S}:=\frac{\mathcal{F}_{_1}|_{_{p\times S}}\oplus \mathcal{F}_{_2}|_{_{p\times S}}}{G}$. Consider the exact sequence-
 \[0\to \mathcal{F}\to \mathcal{F}_{_1}\oplus \mathcal{F}_{_2}\to \mathcal{L}_{_S}\to 0.\]  Since, $\mathcal{F}_{_1}\oplus \mathcal{F}_{_2}$ and $\mathcal{L}_{_S}$ are both flat over $S$. Hence $\mathcal{F}$ is flat over $S$.
\end{remark}
In \cite[Theorem 5.3]{ns} it is shown that there is a smooth, irreducible projective variety which has the coarse moduli property for family of semistable triple of type $I$. We denote this space by $M_{_{12}}$. By the same construction one can construct another smooth, irreducible, projective variety which has the coarse moduli property of semistable triples of type $(II)$. We denote this space by $M_{_{21}}$.
Let \[D_{_1}:=\{[(F_{_1},F_{_2},A)\in M_{_{12}}\mid rk(A)=1\}.\] and \[D_{_2}:=\{[(F'_{_1},F'_{_2},B)]\in M_{_{21}}\mid rk(B)=1\}.\] 
Then, by \cite[Theorem 6.1]{ns} it follows $D_{_1}$ (resp. $D_{_2}$) is a smooth divisor in $M_{_{12}}$(resp. $M_{_{21}}$). 
Now if $(F_{_1},F_{_2},A)\in \stackrel{\to}{C}$ and $rk(A)=1$,
then by Remark \ref{direction}, we get a unique triple $(F'_{_1},F'_{_2},B)\in \stackrel{\leftarrow}{C}$ 
 such that $rk(B)=1$ and $\chi(F'_{_1})=\chi(F_{_1})+1$, $\chi(F_{_2}')=\chi(F_{_2})-1$. Therefore, this association defines a natural isomorphism between $D_{_1}$ and $D_{_2}$. Let us denote this isomorphism by $\Psi$ and $M_{_0}$ be the variety obtained by identifying the closed subschemes $D_{_1}$ and $D_{_2}$ via the isomorphism $\Psi$. Now by Remark \ref{family}  we get a morphism $f_{_1}:M_{_{12}}\to M(2,a,\chi)$ (resp. $f_{_2}:M_{_{21}}\to M(2,a,\chi)$) by associating a triple $(F_{_1},F_{_2},A)$ to the corresponding torsion free sheaf $F$. Clearly $f_{_1}$ and $f_{_2}$ are compatible with the gluing morphism $\Psi$. Thus we get a morphism $M_{_0}\to M(2,a,\chi)$. This morphism is bijective. Also this morphism induces an isomorphism on the dense open subvariety of $M_{_0}$ consisting of rank $2$ triples. Therefore it is a birational morphism. Thus by \cite[Theorem 2.4]{Vit} the variety $M_{_0}$ is isomorphic to the moduli space $M(2,a,\chi)$ as the latter space is projective and seminormal, being the union of two smooth projective variety intersecting transversally, without any one dimensional component.
 
Let $S$ be a finite type scheme and  $\chi_{_i}'=\chi_{_i}-(1-g_{_i})$. Given a family of type $(I)$, $(a_{_1},a_{_2})$ semistable triples $(\mathcal{F}_{_1},\mathcal{F}_{_2},\mathcal{A})$ parametrised by $S$ we get two families of line bundles  $\wedge^2\mathcal{F}_{_i}$ over $X_{_i}\times S$, $i=1,2$. Thus by the universal property of $J^{\chi_{_i}'}(X_{_i})$ we get a morphism \[det_{_1}:M_{_{12}}\to J_{_0}:=J^{\chi_{_1}'}(X_{_1})\times J^{\chi_{_2}'}(X_{_2}).\] such that $det_{_1}((F_{_1},F_{_2},A))=(\wedge^2F_{_1},\wedge^2F_{_2})$ for all closed points $(F_{_1},F_{_2},A)\in M_{_{12}}$.
Similarly, we get another morphism:
\[det_{_2}:M_{_{21}}\to J_{_0}':=J^{\chi_{_1}'+1}(X_{_1})\times J^{\chi_{_2}'-1}(X_{_2}).\] such that $det_{_2}((F_{_1},F_{_2},A))=(\wedge^2F_{_1},\wedge^2F_{_2})$ for all closed points $(F_{_1},F_{_2},A)\in M_{_{21}}$. 
\begin{lemma}\label{det3}
 The fibres of $det_{_i}$ are smooth and the fibres of $det_{_i}$ intersect $D_{_i}$ transversally, $i=1,2$.
 \begin{proof}
 The group $J^0(X_{_1})\times J^{0}(X_{_2})$ acts on $M_{_{21}}$ (resp. $M_{_{21}}$) by $(F_{_1},F_{_2},A)\mapsto (F_{_1}\otimes L_{_1},F_{_2}\otimes F_{_2}\otimes L_{_2},A)$ and on $J_{_0}$ (resp. $J_{_0}'$ ) by $(M_{_1},M_{_2})\mapsto (M_{_1}\otimes L_{_1},M_{_2}\times L_{_2})$ where $(L_{_1},L_{_2})\in J^0(X_{_1})\times J^{0}(X_{_2})$. The morphism $det_{_1}$ (resp. $det_{_2}$) is clearly compatible with the above actions. Thus $det_{_1}$ (resp. $det_{_2}$) is smooth. As $M_{_{12}}$ (resp. $M_{_{21}}$) and $J_{_0}$ (resp. $J'_{_0}$) are smooth, the fibres of $det_{_1}$ (resp. $det_{_2}$) are smooth. Clearly, the divisor $D_{_1}$ (resp. $D_{_2}$) is invariant under the above action. Therefore, $det_{_i}|_{_{D_{_i}}}$ are smooth, $i=1,2$. Thus the fibres of $det_{_i}|_{_{D_{_i}}}$ are also smooth. Clearly, the intersection of a fibre of $det_{_i}$ with $D_{_i}$ is the fibre of $det_{_i}|_{_{D_{_i}}}$. Hence we are done.
 \end{proof}
\end{lemma}
Fix $\xi=(\xi_{_1},\xi_{_2})\in J^{\chi_{_1}'}(X_{_1})\times J^{\chi_{_2}'}(X_{_2})$. Let $det_{_1}^{-1}(\xi):=M_{_{12}}^{\xi}$ and $det_{_2}^{-1}(\xi'):=M_{_{21}}^{\xi'}$ where $\xi'=(\xi(p),\xi(-p))$. By Lemma \ref{det3} the fibre $det_{_1}^{-1}(\xi)$ (resp. $det_{_2}^{-1}(\xi')$) intersects $D_{_1}$ (resp. $D_{_2}$) transversally. Hence $D_{_1}^{\xi}:=det_{_1}^{-1}(\xi)\cap D_{_1}$ and $D_{_2}^{\xi'}:=det_{_2}^{-1}(\xi')\cap D_{_2}$. Let $M_{_{0,\xi}}$ be the closed  subvariety of $M_{_0}$ obtained by gluing $M_{_{12}}^{\xi}$ and $M_{_{21}}^{\xi'}$ along the closed subschemes $D_{_1}^{\xi}$ and $D_{_2}^{\xi'}$ via the isomorphism $\Psi$.

Let $det$ be the morphism defined in Proposition \ref{determin}. We can easily show that $det^{-1}(\xi)$, $\xi\in J_{_0}$ is isomorphic to the variety $M_{_{0,\xi}}$. In the next section we will compute some of the cohomology groups of $M_{_{0,\xi}}$.
\subsubsection{Notation}
Henceforth, we will denote by $M_{_{0,\xi}}$, the moduli space of rank $2$, $(a_{_1},a_{_2})$- semistable bundles with $det\simeq \xi$ and its components by $M_{_{12}}$ and $M_{_{21}}$. We also denote the smooth divisor $D_{_1}^{\xi}$ in $M_{_{12}}$ by $D_{_1}$ and the smooth divisor $D_{_2}^{\xi}$ in $M_{_{21}}$ by $D_{_2}$. 

We conclude this section by proving a geometric fact about the moduli space $M_{_{12}}$ (resp. $M_{_{21}}$).
\begin{lemma}
The moduli space $M_{_{12}}$ (resp, $M_{_{21}}$) is a unirational variety.
\end{lemma}
\begin{proof}\label{unirational}
  To prove $M_{_{12}}$ is unirational we can assume, after tensoring by line bundles, it consists of all triples $(F_{_1},F_{_2},A)$, where $F_{_i}$ is semistable over $X_{_i}$ such that $deg(F_{_i})>2(2g_{_i}-1)$  $i=1,2$. Then, any such $F_{_i}$ can be obtained as an extension: \[0\to \mathcal{O}_{_{X_{_i}}} \to F_{_i}\to \xi_{_i}\to 0,\] where $\xi_{_i}=det(F_{_i})$ for $i=1,2$. The exact sequences of this type are classified by $V_{_{\xi_{_i}}}:=Ext^1(\mathcal{O}_{_{X_{_i}}},\xi_{_i})=H^1(X_{_i},\xi_{_i}^*)$. Let $\mathcal{E}_{i}$ be the universal extension over $X_{_i}\times V_{_{\xi_{_i}}}$. We denote the restriction ${\mathcal{E}_{_i}}_{_{|p\times V_{_{\xi_{_i}}}}}$ by ${\mathcal{E}_{_i}}_{_p}$. Clearly, $Hom(\mathcal{E}_{_1},\mathcal{E}_{_2})$ parametrises a family of triples in the sense we have defined family of triples and if $(F_{_1},F_{_2},A)$ be a triple corresponding to the closed point $A\in Hom(\mathcal{E}_{_1},\mathcal{E}_{_2})$ then $F_{_i}$'s are the extensions of the type described before. Now as $F_{_i}$'s are semistable if we choose an isomorphism $A:F_{_1}(p)\to F_{_2}(p)$ then by Lemma \ref{semistable}, $(F_{_1},F_{_2},A)$ is semistable. Thus we conclude the set of points $W$ where the corresponding triple is semistable is a nonempty Zariski open set of $Hom(\mathcal{E}_{_1},\mathcal{E}_{_2})$. Therefore, by the 
coarse moduli property of $M_{_{
12}}$, we 
get a  morphism from $W$ to $M_{_{12}}$. Clearly the morphism $W\to M_{_{12}}$ is surjective. Hence, $M_{_{12}}$ is a unirational variety. The same argument shows the moduli space $M_{_{21}}$ is also a unirational variety.
  \end{proof}
\section{Topology of \texorpdfstring{$M_{_{0, \xi}}$}{}}\label{top}
In this section, our main aim is to outline a strategy to compute the  cohomology groups of $M_{_{0,\xi}}$ and compute explicitly the third cohomology group. We make the following convention:
Let $X$ be a topological space. By $H^k(X)$ we mean the cohomology groups of $X$ with the coefficients in $\Q$, $k\geq 0$. Whenever we obtain any results for other coefficients, e.g $\Z$, we will specifically mention it. Suppose $X$ and $Y$ be varities over $\C$. Whenever we say $X\to Y$ a topological fibre bundle, we assume the underlying topology of $X$ and $Y$ to be complex analytic topology.

Let $Y$ be a smooth,projective curve of genus $g_{_Y}\geq 2$ and $M_{_Y}$ be the moduli space of rank $2$ semistable bundles with fixed determinant. The cohomology groups of $M_{_Y}$ are quite well studied in the literature. When the determinant is odd $M_{_Y}$ is a smooth projective variety of dimension $3g_{_Y}-3$ and the cohomology groups with integral coefficients are completely known. When the determinant is even $M_{_Y}$ need not be smooth. In fact it is known that the singular locus of $M_{_Y}$ is precisely the complement $M_{_Y}\setminus M_{_Y}^{s}$ if $g_{_Y}\geq 3$ where $M_{_Y}^s$ is the open subset consisting of stable bundles (see \cite[Theorem 1]{NAR}). In this case also the Betti numbers are determined in the work of \cite{bs}. We will summarize some of the results concerning the cohomology groups of $M_{_Y}$ in both the cases i.e odd determinant and even determinant:
\begin{lemma}\label{topology}
\begin{enumerate}
 \item Let $M_{_Y}$ be the moduli space of rank $2$ semistable bundles with odd determinant. Then $M_{_Y}$ is a smooth, projective rational variety (\cite{ne1}) and hence it is simply connected and $H^3(M_{_Y},\Z)_{_{tor}}=0$. Furthermore, $b_{_1}(M_{_Y})=0$, $b_{_2}(M_{_Y})=1$, $b_{_3}(M_{_Y})=2g_{_Y}$, where $b_{_i}$ are the Betti numbers (\cite{ne}).
 \item Let $M_{_Y}$ be the moduli space of rank $2$ semistable bundles with even determinant. Then $M_{_Y}^s$ is a simply connected variety (\cite[Proposition 1.2]{bbgn}). Furthermore, we have $b_{_1}(M_{_Y})=0$, $b_{_2}(M_{_Y})=1$ and $b_{_3}(M_{_Y}^s)=2g_{_Y}$, where $b_{_i}$ are the Betti numbers (\cite{N}, \cite[Section 3]{bs}).
 \end{enumerate}
\end{lemma}

Let $M_{_1}$ (resp. $M_{_1}'$) be the moduli space of rank $2$, semistable bundles over $X_{_1}$  with $det\simeq \xi_{_1}$ (resp. with $det\simeq \xi_{_1}(p)$) and $M_{_2}$ (resp. $M_{_2}'$) be the moduli space of rank $2$, semistable bundes over $X_{_2}$ with $det\simeq \xi_{_2}$ (resp. $\det\simeq \xi_{_2}(-p)$)  where $\xi_{_i}$'s are line bundles of degree $d_{_i}=\chi_{_i}-2(1-g_{_i})$ for $i=1,2$ and the integers $\chi_{_1}$, $\chi_{_2}$ satisfy the inequality \ref{inq1}. Since  $\chi$ is odd, one of the integer in the pair $(d_{_1},d_{_2})$ is odd and the other is even. We assume that $d_{_1}$ is odd and $d_{_2}$ is even. Therefore, $M_{_1}$ and $M_{_2}'$ are smooth projective varieties. Let $M_{_2}^s$ be the open subvariety of $M_{_2}$ consisting of all the isomorphism classes 
of stable bundles over $X_{_2}$ and $M_{_1}'^s$ be the open subvariety of $M_{_1}'$ consisting of all the isomorphism classes of stable bundles over $X_{_1}$. Note that $M_{_2}\setminus M_{_2}^s$ is precisely the singular locus of $M_{_2}$ if $g_{_2}\geq 3$ and $M_{_1}'\setminus M_{_1}'^s$ is precisely the singular locus of $M_{_1}'$ if $g_{_1}\geq 3$.

Let us denote the open subvariety $M_{_1}\times M_{_2}^s$ of $M_{_1}\times M_{_2}$ by $B$.
We will show the following,
\begin{proposition}\label{morphism}
 There is a surjective morphism $p: M_{_{12}}\to M_{_1}\times M_{_2}$. Moreover, $p:P\to B$ is a topological
 $\bp^3$-bundle where $P:=p^{-1}(B)$. 
  \begin{proof}
  Let $S$ be a finite type scheme and $(\mathcal{F}_{_1},\mathcal{F}_{_2},\mathcal{A})$ be a family of triples parametrised by $S$ such that $({\mathcal{F}_{_1}}_{_s},{\mathcal{F}_{_2}}_{_s},\mathcal{A}_{_s})$ is $(a_{_1},a_{_2})$-semistable of type $(I)$ for all $s\in S$ where ${\mathcal{F}_{_i}}_{_s}:=\mathcal{F}_{_i}|_{_{X_{_0}\times s}}$. We also assume $\wedge^2{\mathcal{F}_{_i}}_{_s}\simeq \xi_{_i}$, $i=1,2$. Then by Fact \ref{semistable} ${\mathcal{F}_{_i}}_{_s}$, $i=1,2$, are semistable for all $s\in S$. Thus we get a morphism $p:M_{_{12}}\to M_{_1}\times  M_{_2}$. Let $([F_{_1}],[F_{_2}])\in M_{_1}\times M_{_2}$. Choose any isomorphism $A:F_{_1}(p)\to F_{_2}(p)$. Then, by Fact \ref{semistable1} $(F_{_1},F_{_2},A)$ is $(a_{_1},a_{_2})$-semistable. Therefore, $p$ is surjective.

  Now we show that $p:P\to B$ is a topological $\bp^3$-bundle. Let $b=([F_{_1}],[F_{_2}])\in B$. Our first claim is the fibre $p^{-1}(b)$ is homeomorphic to $\bp Hom(F_{_1}(p),F_{_2}(p))\simeq \bp^3$. Let $A\in Hom(F_{_1}(p),F_{_2}(p))$ and $A\neq 0$. Since both $F_{_i}$'s are stable, by Lemma \ref{semistable1}, $(F_{_1},F_{_2},A)$ is $(a_{_1},a_{_2})$-stable. Thus we get a morphism $i_{_b}:Hom(F_{_1}(p),F_{_2}(p))\setminus 0\to M_{_{12}}$. Clearly, $i_{_b}(Hom(F_{_1}(p),F_{_2}(p)))\setminus 0)=p^{-1}(b)$. Note that $(F_{_1},F_{_2},A)$ and $(F_{_1},F_{_2},\lambda A)$ are isomorphic for all $\lambda\in \C^*$. Thus $i_{_b}$ descends to a morphism $i_{_b}:\bp Hom(F_{_1}(p),F_{_2}(p))\to p^{-1}(b)$. Now we show that $i_{_b}$ is  injective. Then the claim will follow. Let  $A, B\in \bp Hom(F_{_1}(p),F_{_2}(p))$ are disticnt points. Then the triples $(F_{_1},F_{_2},A)$, $F_{_1},F_{_2},B)$ are non isomorphic. Suppose, $(F_{_1},F_{_2},A)$ and $(F_{_1},F_{_2},B)$ are isomorphic as triples. Then there are isomorphisms $\phi_{_i}:F_{_i}\to F_{_i}$, $i=1,2$ such that we have the following commutative diagram:
  \begin{equation}
                     \xymatrix{F_{_1}(p) \ar[r]^{\phi_{_1}(p)} \ar[d]^{A} & F_{_1}(p)
                     \ar[d]^{B} \\
                     F_{_2}(p)\ar[r]^{\phi_{_2}(p)} &  
                     F_{_2}(p)}
                    \end{equation}
      Since $F_{_i}$ are stable the only isomorphisms of $F_{_i}$ are $\lambda I$ for some scalor $\lambda$. Thus we have $\phi_{_i}(p)=\lambda_{_i} I$, $i=1,2$. From the commutativity of the above diagram we get $B\lambda_{_1}=\lambda_{_2}A$. Thus $B=\lambda_{_1}^{-1}\lambda_{_2}A$. Hence a contradiction as $A$ and $B$ are disticnt in $\bp Hom(F_{_1}(p),F_{_2}(p))$. Therefore, the morphism is injective. Since the fibres of $p:P\to B$ are compact, $p:P\to B$ is a proper, analytic map. 
      
      Our next claim is that the induced map $dp:T_{_F}\to T_{_{p(F)}}$ at the level of Zariski tangent space is surjective for all $F=(F_{_1},F_{_2},A)\in P$. Let $(F_{_1},F_{_2})\in B$. Since $F_{_i}$ are both stable, $i=1,2$, the Zariski tangent space $T_{_{F_{_i}}}\simeq H^1(End(F_{_i}))_{_0}$ where $H^1(End(F_{_i}))_{_0}=Ker(tr^1:H^1(End(F_{_i}))\to H^1(\mathcal{O}_{_{X_{_i}}}))$ and $tr^1$, the trace homomorphism (see \cite[Theorem 4.5.4]{huy}). Thus the tangent space $T_{_{(F_1,F_2)}}B\simeq H^1(End(F_{_1}))_{_0}\times H^1(End(F_{_2}))_{_0}$. Now a cocycle in $H^1(End(F_{_i}))$ corresponds to a locally free sheaf $\mathcal{F}_{_i}$ over $X_{_i}\times D$ such that ${\mathcal{F}_{_i}}_{t_0}\simeq F_{_i}$ where $D=Spec\frac{\C[\varepsilon]}{\varepsilon^2}$ and $t_{_0}=(\varepsilon)$. Choose an isomorphism $\mathcal{A}:F_{_1}(p)\to F_{_2}(p)$. Then clearly, $A$ lifts to a $\mathcal{O}_{_D}$-module homomorphism $\mathcal{A}:F_{_1}|_{_{p\times D}}\to F_{_2}|_{_{p\times D}}$. Thus we get a triple $(\mathcal{F}_{_1},\mathcal{F}_{_2},\mathcal{A})$ parametrised by $D$ such that $({\mathcal{F}_{_1}}_{t_0},{\mathcal{F}_{_2}}_{t_0},\mathcal{A}_{_{t_0}})=(F_{_1},F_{_2},A)$. By the coarse moduli property of $M_{_{12}}$ we get a morphism $x:D\to M_{_{12}}$ such that $x(t_{_0})=(F_{_1},F_{_2},A)$. In otherwords we get a point of the Zariski tangent space at $(F_{_1},F_{_2},A)$. Thus $dp$ is surjective. Therefore, $p$ is a proper, surjective,  holomorphic submersion. Hence, $p:P\to B$ is a topological $\bp^3$- bundle.
  \end{proof}
\end{proposition}

\begin{remark}\label{2}
 By the same arguments  as before  we get a morphism $p':M_{_{21}}\to M'_{_1}\times M'_{_2}$, where $M'_{_1}$ is the moduli space of rank $2$ semistable bundles over $X_{_1}$ with $det\simeq \xi'_{_1}$ and $M'_{_2}$ be the moduli space of rank $2$ semistable bundles over $X_{_2}$ with $det\simeq \xi'_{2}$, where $\xi'_{_1}:=\xi_{_1}(p)$ and $\xi'_{_2}:=\xi_{_2}(-p)$. Moreover, $p': \overline P\to B'$ is a $\bp^3$-fibration where $B'=M_{_1}^s\times M_{_2}$ and $\overline P= p'^{-1}(B')$.
\end{remark}
\subsection{Codimension computations}
 
 In the following proposition we compute the codimension of the complement of the open subvariety $P$ in $M_{_{12}}$ (resp. the complement of $\overline P$ in $M_{_{21}}$).
 
 Let $K'$ denote the complement of $P$ in $M_{_{12}}$ and $K_{_2}'$ denote the complement of $\overline P$ in $M_{_{21}}$. Then we have
 \begin{proposition}\label{codimension}{\ }

\noindent  $(a)$ $Codim (K',M_{_{12}})= g_{_2}-1.$ where $g_{_2}$ is the genus of $X_{_2}$\\ 
  $(b)$ $Codim (K_{_2}',M_{_{21}})= g_{_1}-1$ where $g_{_1}$ is the genus of $X_{_1}$.
  \end{proposition}
  \begin{proof}

%$Codim(K',M_{_{12}})=\mbox{dim}M_{_{12}}-\mbox{dim}K'$.
We will only show $(a)$. The proof of $(b)$ is similar.
Note that if $(F_{_1},F_{_2},A)\in K'$ then $F_{_2}$ is a strictly semistable bundle on $X_{_2}$. Therefore, $K'=p^{-1}(M_{_1}\times K)$, where $K=M_{_2}\setminus M_{_2}^s$ and $p$ is the morphism as in Proposition \ref{morphism}. . Now $F_{_2}\in K$ if and only if there is a short exact sequence \[0\to L_{_2}\to F_{_2}\to L_{_1}\to 0,\] for some line bundles $L_{_1}$, $L_{_2}$ 
with $deg(L_{_i})=\frac{d_{_2}}{2}$, $i=1,2$. Clearly, $L_{_1}\otimes L_{_2}\simeq det(F_{_2})\simeq  \xi_{_2}$. Thus $K$ consists of all S-equivalance classes $[L_{_1}\oplus L_{_2}]$ of semistable bundles on $X_{_2}$ where $L_{_1},L_{_2}\in J^{d'_{_2}}(X_{_2})$, $d_{_2}'=\frac{d_{_2}}{2}$ such that $L_{_1}\otimes L_{_2}\simeq \xi_{_2}$. Let $K^{0}$ be the subset of $K$ consisting of all S- equivalence classes $[L_{_1}\oplus L_{_2}]$ such that $L_{_1}\ncong L_{_2}$. Then by \cite[Lemma 4.3]{NAR} $K^{0}$ is an open and dense subset of $K$. Let $K''=p^{-1}(M_{_1}\times K^{0})$. Then $K''$ is open and dense in $K'$. Therefore, we get $dim(K')=dim(K'')$. Now we will find a parameter variety of isomorphism classes of all $(a_{_1},a_{_2})$-semistable triples $(F_{_1},F_{_2},A)$ where $F_{_1}\in M_{_1}$ and $F_{_2}\in \bp Ext^1(L_{_1},L_{_2})$ for some $L_{_i}\in J^{d'_{_2}}(X_{_2})$, $i=1,2$ with $L_{_1}\ncong L_{_2}$ and show that this parameter variety has same dimension as $K''$. 

Let $J^{\xi_{_2}}=\{(L_{_1},L_{_2})\in J^{d'_{_2}}(X_{_2})\times J^{d'_{_2}}(X_{_2})\mid L_{_1}\otimes L_{_2}\simeq \xi_{_2}\}$. Note that $J^{d'_{_2}}(X_{_2})$ is isomorphic to $J^{\xi_{_2}}$ by $L\mapsto (L,\xi_{_2}\otimes L^{-1})$. Therefore, $J^{\xi_{_2}}$ is a closed subvariety of $J^{d'_{_2}}(X_{_2})\times J^{d'_{_2}}(X_{_2})$ of dimension $g_{_2}$. Let $J'=\{(L_{_1},L_{_2})\in J^{d_{_2}}(X_{_2})\times J^{d'_{_2}}(X_{_2})\mid L_{_1}\ncong L_{_2}\}$. Then, clearly $J'$ is an open and dense subvariety of $J^{d'_{_2}}(X_{_2})\times J^{d'_{_2}}(X_{_2})$. Let $J'^{\xi_{_2}}:=J'\cap J^{\xi_{_2}}$. 

We will construct a projective bundle $\bp$ over $J'$ such that the fibre over a point $(L_{_1},L_{_2})\in J'$ is isomorphic to $\bp Ext^1(L_{_1},L_{_2})$: Let $\mathcal{L}$ be the Poincare line bundle over $X_{_2}\times J^{d'_{_2}}(X_{_2})$ and $\mathcal{L}_{_i}:=(id\times p_{_i})^{*}\mathcal{L}$ where $p_{_i}:J^{d'_{_2}}(X_{_2})\times J^{d'_{_2}}(X_{_2})\to J^{d'_{_2}}(X_{_2}) $ is the $i$th projection for $i=1,2$. Then $V:=R^1{p_{_{J'}}}_{_*}Hom(\mathcal{L}_{_1},\mathcal{L}_{_2})$ is a locally free sheaf of rank $g_{_2}-1$ over $J'$ where $p_{_{J'}}:X_{_2}\times J'\to J'$ is the projection. Let $\bp$ over $J'$ be the projective bundle associated  to $V$. Then the fibre over a point $(L_{_1},L_{_2})\in J'$ is isomorphic to $\bp Ext^1(L_{_1},L_{_2})$. Let $P'=\bp|_{_{J'^{\xi_{_2}}}}$.  Let $\mathcal{G}$ be the universal extension over $X_{_2}\times P'$ (see \cite[Proposition 3.1]{NAR}) and $F$ be a universal bundle over $X_{_2}\times M_{_1}$ (note that $F$ exists as the degree and rank of the vector bundles in $M_{_1}$ are coprime).

Let $\mathcal{G}_{_p}:=\mathcal{G}|_{_{_{p\times P'}}}$ and $F_{_p}:=F|_{_{_{p\times M_{_1}}}}$. Clearly, $Hom(F_{_p},G_{_p})$ parametrises a family of triples of type $(I)$ such that every closed point in  $ \ Hom(F_{_p},G_{_p})$ corresponds to a triple $(F_{_1},F_{_2},A)$ where $F_{_1}\in M_{_1}$ and $F_{_2}\in P'$. Note that if $E\in \bp Ext^1(L_{_1},L_{_2})$ then $Aut(E)\simeq \C^*$ whenever $L_{_1}\ncong L_{_2}$ (see \cite[Lemma 4.1]{NAR}). Let $A,B\in \bp Hom(F_{_p},\mathcal{G}_{_p})$ be two distinct closed points and $(F_{_1},F_{_2},A)$, $(G_{_1},G_{_2},B)$ be the corresponding triples. Then $(F_{_1},F_{_2},A)$ and $(G_{_1},G_{_2},B)$ are non isomorphic. This follows from the two facts: if $E_{_1}\in M_{_1}$ and $E_{_2}\in P'$ then $Aut(E_{_i})\simeq \C^*$.  If $E_{_1},E_{_2}\in \bp Ext^1(L_{_1},L_{_2})$ are distinct then $E_{_1}$ and $E_{_2}$ are non isomorphic (  \cite[Lemma 3.3]{NAR}).  Let $K_{_1}$ be the subset of $\bp Hom(F_{_p},\mathcal{G}_{_p})$ whose closed points correspond to the triples $(F_{_1},\mathcal{G}_{_2},A)$ such that $rk(A)=2$. Then $K_{_1}$ is an open subset in $\bp Hom(F_{_p},\mathcal{G}_{_p})$. Note that by  Lemma \ref{semistable1}, any closed point of $K_{_1}$ is semistable. Therefore, by the coarse moduli property of $M_{_{12}}$, we get a morphism  $i_{_K}:K_{_1}\to M_{_{12}}$. By the above discussions  $i_{_K}$ is injective. Clearly, the image $i_{_K}(K_{_1})$ is dense in $K''$ since if $(F_{_1},F_{_2},A)\in K''\setminus i_{_K}(K_{_1})$ then $F_{_2}\simeq L_{_1}\oplus L_{_2}$ for some $L_{_1},L_{_2}\in J^{d_{_2}'}(X_{_2})$. Therefore, $dim(K_{_1})= dim(K'')$
%\simeq L_{_1}\oplus L_{_2}$ for some $L_{_1},L_{_2}\in J^{d'_{_2}}(X_{_2})\times J^{d'_{_2}}(X_{_2})$. Therefore, $i_{_K}(K_{_1})$ is dense in $K''$. Consider the natural action $\C^*\times \C^*$ on $Hom(F_{_p},\mathcal{G}_{_p})\setminus 0$. Clearly $K_{_1}$ is invariant under this action and $i_{_K}$ is $\C^*\times \C^*$ invariant, by the above discussions. Hence, $i_{_K}$ induces an injective morphism $i_{_K}:K_{_1}'\to K''$ where $K_{_1}'=K_{_1}/\C^*\times \C^*$. Hence, we have $dim(K_{_1}'')=dim(K'')$

 We have $\mbox{dim}(M_{_1})=3g_{_1}-3$ and $\mbox{dim}(P')=2g_{_2}-2$. Therefore, $\bp Hom(F_{_p},\mathcal{G}_{_p})=dim(K_{_1})=3g_{_1}-3+2g_{_2}-2+3=3g_{_1}+2g_{_2}-2$.  
 Note that $P$ is an dense open subvariety of $M_{_{12}}$, therefore $dim(P)=dim(M_{_{12}})$. Now, from the proof of Proposition \ref{morphism}, $p:P\to B$ is flat with fibres isomorphic to $\bp^3$ as  algebraic varieties. Therefore, $dim(P)=dim(B)+3=3(g_{_1}+g_{_2})-3$. Hence, we have $\mbox{dim} M_{_{12}}=3g_{_1}+3g_{_2}-3$. Since 
$Codim(K',M_{_{12}}) = dim(M_{_{12}})-dim(K')$ and $dim(K_{_1})= dim(K'')$, we see
 \begin{equation*}
\begin{split} 
 & Codim(K',M_{_{12}}) = 3g_{_1}+3g_{_2}-3-3g_{_1}-2g_{_2}+2\\
 & \hspace{3cm} =g_{_2}-1. \qedhere
\end{split}
\end{equation*}
 \end{proof}
Now we recall a well-known fact (see \cite[Lemma 12]{bs}).
\begin{lemma}\label{excision}
  Let $X$ be a smooth projective variety and $k:=Codim(X/U)$, where $U$ be an open subset of $X$ .
  Then we have $H^i(X,\Z)\simeq H^i(U,\Z)$ for all $i< 2k-1$.
  \end{lemma}
  Using the above Lemma and Proposition \ref{codimension} we immediately get the following
  \begin{proposition}\label{imp}
  With the above notations,
  \begin{item}
  \item[(i)] $H^i(M_{_{12}},\Z)\simeq H^i(P,\Z)$ for $i<2k-1$ where $k=g_{_2}-1$.
  \item[(ii)] $H^i(M_{_{21}},\Z)\simeq H^i(\overline P,\Z)$ for $i<2k'-1$ where $k'=g_{_1}-1$.
  \end{item}
  \end{proposition}
 \subsection{Computation of cohomology groups of \texorpdfstring{$M_{_{0, \xi}}$}{}}
In this subsection we will outline the strategy to compute the Betti numbers of the component $M_{_{12}}$(resp. $M_{_{21}}$) and compute the third cohomology of $M_{_{0,\xi}}$ in full details.
First we compute the Betti numbers of $P$ (resp. $\overline P$) using Leray-Hirsh Theorem:
\begin{theorem}\label{LH}(Leray-Hirsh)
Let $f:X\to Y$ be a topological fibre bundle with fibres isomorhic to $F$. Suppose, $e_{_1},\cdots,e_{_n}\in H^*(X)$ such that $H^*(X_{_y})$ is freely generated by $i_{_y}^*e_{_1},\cdots, i_{_y}^*e_{_n}$ for all $y\in Y$ where $X_{_y}=f^{-1}(y)$ and $i_{_y}:X_{_y}\to X$ is the inclusion. Then $H^*(X)$ is freely generated as a $H^*(Y)$-module by $e_{_1},\cdots,e_{_n}$.
\end{theorem}
\begin{proposition}\label{lbetti}
The $k$-th Betti number $b_{_{k}}(P)=\sum_{_{l+m=k}}b_{_l}(B)b_{_m}(\bp^3)$ (resp. $b_{_{k}}(\overline P)=\sum_{_{l+m=k}}b_{_l}(\overline B)b_{_m}(\bp^3)$.
%The variety $P$ (resp. $\overline{P}$) is simply connected. Moreover, the second Betti number $b_{_2}(P)=3$ (resp. $b_{_2}(\overline P)=3$) and the third Betti number $b_{_3}(P)=2g$ (resp. $b_{_3}(\overline P)=2g$) where $g$ is the arithmatic genus of $X$.
 \begin{proof}
  Since  $P$ and $B$ are both smooth varieties and $p:P\to B$ is a submersion we get $p$ is smooth. Therefore, the fibres of $p|_{_P}$ are smooth. From the proof of Proposition \ref{morphism}, it follows that the fibres of $p$ are isomorphic to $\bp^3$ as algebraic varieties. Choose a relatively ample line bundle $L$ over $P$. Now $p_{_*}L$ is locally free by Zariski Main theorem. Therefore,  we get that the dimension of $H^0(p^{-1}(b),L|_{_{p^{-1}(b)}})$ is constant for all $b\in B$. Hence, $L|_{_{p^{-1}(b)}}=\mathcal{O}(k)$ for some $k>0$ for all $b\in B$. Consider the cohomolgy classes $c_{_1}(L)$, $c_{_1}(L)^2$, $c_{_1}(L)^3$. We denote by $j_{_b}:p^{-1}(b)\to P$ the inclusion. Then $H^*(p^{-1}(b))$ is freely generated by $j_{_b}^*c_{_1}(L)$, $j_{_b}^*c_{_1}(L)^2$ and $j_{_b}^*c_{_1}(L)^3$ for all $b\in B$. Thus using Leray-Hirsch theorem we get: \[b_{_k}(P)=\sum_{_{l+m=k}}b_{_l}(B)b_{_m}(\bp^3).\] where $b_{_k}(X)$ denotes the $k$th Betti number of a space $X$. %By Lemma \ref{topology}, \ref{topology2} and the Kunneth fomula we get that $b_{_1}(B)=0$, $b_{_2}(B)=3$ and $b_{_3}(B)=2g$. Also $b_{_k}(\bp^3)=1$ if $k$ is even and $0$ otherwise. Thus, by the above observation we get that $b_{_1}(P)=0$, $b_{_2}(P)=3$ and $b_{_3}(P)=2g_{_1}$. Moreover $P$ is simply connected as $B$ is simply connected by Lemma \ref{topology}, \ref{topology1} and Kunneth formula.
 \end{proof}
\end{proposition}
As a corollary of the above Proposition we immediately get:
\begin{corollary}\label{betti}
 $(i)$ $b_{_1}(P)=0$ (resp.$b_{_1}(\overline P)=0$) , $(ii)$ $b_{_2}(P)=3$ (resp. $b_{_2}(\overline P)=3$) and $(iii)$ $b_{_3}(P)=2g$ (resp. $b_{_3}(\overline P)=2g$) where $g$ is the arithmatic genus of $X_{_0}$ and $b_{_i}$'s are the Betti numbers, $i=1,2,3$.
 \begin{proof}
  By Lemma \ref{topology} and the Kunneth formula it follows that $b_{_1}(B)=b_{_1}(M_{_1})+b_{_1}(M_{_2}^s)=0$, $b_{_2}(B)=b_{_2}(M_{_1})+b_{_2}(M_{_2}^s)=1+1=2$ and $b_{_3}(B)=b_{_3}(M_{_1})+b_{_3}(M_{_2}^s)=2g_{_1}+2g_{_2}=2g$. Thus by Proposition \ref{lbetti} we get $b_{_1}(P)=0$, $b_{_2}(P)=3$ and $b_{_3}(P)=2g$.
  \end{proof}
\end{corollary}

\begin{remark}
 By above proposition all the Betti numbers of $P$ can be computed using the above argument as the Betti numbers of the varities $M_{_1}$ and $M_{_2}^s$ are well known (see \cite[page 113]{bs}).
\end{remark}
Let $g_{_1},g_{_2}>3$. Then as a consequence of Proposition \ref{imp} and Corollary \ref{betti} we immediately get:
\begin{theorem}\label{great}
\begin{enumerate}
With the notations above,
 \item  $H^1(M_{_{12}})=0$ (resp. $H^1(M_{_{21}})=0$). 
 \item $H^2(M_{_{12}})\simeq \Q\oplus \Q \oplus \Q$ (resp. $H^2(M_{_{21}})\simeq \Q\oplus \Q \oplus \Q$).
 \item $H^3(M_{_{12}})\simeq \Q^{2g}$ (resp. $H^2(M_{_{21}})\simeq \Q^{2g}$).
\end{enumerate}
\begin{remark}
 $M_{_{12}}$ (resp. $M_{_{21}}$) is a smooth, projective unirational variety by Lemma \ref{unirational}. Therefore, by a result of Serre (\cite{Se}) $M_{_{12}}$ (resp. $M_{_{21}}$) is a simply connected variety.
\end{remark}

%\begin{proof}
 %$(1)$ follows from the unirationality of $M_{_{12}}$ and the fact the fundamental group of smooth projective uniratinal variety is trivial (see \cite[Lemma 1]{Se}. For the proof of $(2)$ we first observe that since $M_{_{12}}$ is simply connected by universal coefficient theorem $H^2(M_{_{12}},\Z)$ has no torsion. Thus by Corollary \ref{betti} and \ref{imp} we get $(2)$. The proof of $(3)$ follows from Theorem \ref{PP} and \ref{imp}
 %\end{proof}
\end{theorem}
\subsection{Continuation of the cohomology computation}
%Let $P_{_0}=P\cup \overline P$ and $D_{_0}=P\cap \overline P$ where $P$ (resp. $\overline P$) is the open subvarity of $M_{_{12}}$ (resp. $M_{_{21}}$) as defined in the previous section. First we will compute the cohomology groups of $P_{_0}$.  Before we proceed to compute the cohomology group $H^3(P_{_0})$ we will make few more observations.
We have $M_{_{0,\xi}}=M_{_{12}}\cup M_{_{21}}$. Let $D=M_{_{12}}\cap M_{_{21}}$. We will compute the third cohomology group of $M_{_{0,\xi}}$ using Mayer-Vietoris sequence. Before we compute the cohomology group we will make few more observations.

Let  $P_{_1}$ be the moduli space of rank $2$, parabolic semistable bundles $(F_{_1},0\subset F^2F_{_1}\subset F_{_1}(p))$ on $X_{_1}$  with parabolic weights 
 $0<\frac {a_{_1}}{2}<
 \frac {a_{_2}}{2}<1$ and $detF_{_1}\simeq \xi_{_1}$. Let $P_{_2}$ be the moduli space of rank $2$, parabolic semistable bundles $(F_{_2},0\subset F^2F_{_2}(p)\subset F_{_2}(p))$ on $X_{_2}$  with parabolic weights 
 $0<\frac {a_{_1}}{2}<
 \frac {a_{_2}}{2}<1$ and $detF_{_2}\simeq \xi_{_2}$.
 By Fact \ref{para} $(a)$ any parabolic semistable bundle in $P_{_1}$ (resp. in $P_{_2}$) is parabolic stable. Therefore, one can show that $P_{_i}$'s are smooth, $i=1,2$. Since $E\in P_{_i}$ is semistable by Fact \ref{para} $(b)$,  we get morphisms $q_{_i}:P_{_i}\to M_{_i}$, $i=1,2$. Let $P_{_2}^s=q_{_2}^{-1}(M_{_2}^s)$. Thus we get a morphism $q:=(q_{_1},q_{_2}):P_{_1}\times P_{_2}\to M_{_1}\times M_{_2}$ such that $q^{-1}(B)=P_{_1}\times P_{_2}^s$ where $B:=M_{_1}\times M_{_2}^s$. Now by using the same argument given in \cite[Theorem 6.1]{ns} we can show that there is an embedding $i:P_{_1}\times P_{_2}\to M_{_{12}}$ such that the image is isomorphic to $D$. Thus we have a commutative diagram of morphisms:
 \begin{equation}\label{parab}
\xymatrix{
P_{_1}\times P_{_2} \ar[rr]^{i} \ar[rd]_{q} && M_{_{12}}\ar[ld]^{p} \\
&M_{_1}\times M_{_2}
}
\end{equation}
where $p$ is the morphism in Proposition \ref{morphism}.

Let $q_{_1}':P_{_1}\to M_{_1}'$ be the morphism defined by $E_{_1}\to E_{_1}'$  and $q_{_2}':P_{_2}\to M_{_2}'$ defined by $E_{_2}\to E_{_2}'$ where $E_{_i}'$ are the Hecke modifications of $E_{_i}$, $i=1,2$ (see Remark \ref{direction}). Then we get another commutative diagram of morphisms: 
\begin{equation}\label{parab1}
\xymatrix{
P_{_1}\times P_{_2} \ar[rr]^{j} \ar[rd]_{q'} && M_{_{21}}\ar[ld]^{p} \\
&M_{_1}'\times M_{_2}'
}
\end{equation}
where the morphism $q':P_{_1}\times P_{_2}\to M_{_1}'\times M_{_2}'$ is given by the association $(E_{_1},E_{_2})\mapsto (E_{_1}',E_{_2}')$ and $p'$ is the morphism in Remark \ref{2}.

In the following lemma we summarize some topological facts about the moduli spaces $P_{_i}$, $i=1,2$.
\begin{lemma}\label{topology1}{\ }
\noindent$(i)$ $P_{_i}$, $i=1,2$, are smooth, projective and rational variety being  $\bp^1$- bundles associated to algebraic vector bundles over coprime moduli spaces (see Remark \ref{rational}). In particular, $P_{_i}$ are simply connected and $Pic(P_{_i})\simeq H^2(P_{_i},\Z)$ for $i=1,2$.\\
 $(ii)$ $H^1(P_{_i},\Z)=0$, $H^2(P_{_i},\Z)\simeq \Z\oplus \Z$ and $H^3(P_{_i},\Z)\simeq \Z^{2g_{_i}}$ for $i=1,2$ (follows from the previous statement).\\
\end{lemma}

\begin{lemma}\label{third}
 \noindent
  $(1)$ $q^*:H^3(P_{_1}\times P_{_2}^s)\simeq H^3(M_{_1}\times M_{_2}^{s})$.\\
  $(2)$ $p^*:H^3(P)\simeq H^3(M_{_1}\times M_{_2}^s)$.\\
 \begin{proof}
 $(1)$ From the fact \ref{para} $(c)$ it follows that the topological fibre of $q:P_{_1}\times P_{_2}^s\to M_{_1}\times M_{_2}^s$ is $\bp^1\times \bp^1$. Using the similar arguments given in \ref{morphism} and \ref{betti} we can show that $q$ is a topological $\bp^1\times \bp^1$- bundle satisfying the hypothesis of the Leray-Hirsch theorem. Thus by Leray-Hirsch theorem $q^*:H^3(P_{_1}\times P_{_2}^s)\simeq H^3(M_{_1}\times M_{_2}^{s})$.
 
 $(2)$ The proof is already given in \ref{betti}.
\end{proof}
\end{lemma}

\begin{lemma}\label{third1}
 $i^*:H^3(P_{_1}\times P_{_2})\simeq H^3(M_{_{12}})$ where $i:P_{_1}\times P_{_2}\to M_{_{12}}$ is the inclusion.
 \begin{proof}
  First note that $i(P_{_1}\times P_{_2}^s)\subset P$. Therefore, we have-
  \begin{equation}
\xymatrix{
P_{_1}\times P_{_2}^s \ar[rr]^{i} \ar[rd]_{q} && P\ar[ld]^{p} \\
&M_{_1}\times M_{_2}^s
}
\end{equation}
  By the commutativity of the above diagram we get $i^*p^*=q^*$. Since, by Lemma \ref{third}, $p^*$ and $q^*$ are isomorphisms, we get $i^*:H^3(P)\to H^3(P_{_1}\times P_{_2}^s)$ is an isomorphism. By an argument given in \cite[Proposition 7]{Ba1} we can show that $Codim(K,P_{_1}\times P_{_2})=g_{_2}-1$ where $K=P_{_1}\times P_{_2}\setminus P_{_1}\times P_{_2}^s$. Thus, by Lemma \ref{excision}, we get $i_{_1}^*:H^3(P_{_1}\times P_{_2})\to H^3(P_{_1}\times P_{_2}^s)$ is an isomorphism where $i_{_1}:P_{_1}\times P_{_2}^s\to P_{_1}\times P_{_2}$ is the inclusion. Also, we have shown $i_{_2}^*:H^3(M_{_{12}})\to H^3(P)$ is an isomorphism where $i_{_2}:P\to M_{_{21}}$ is the inclusion. Thus $i^*:H^3(M_{_{12}})\to H^3(P_{_1}\times P_{_2})$ is an isomorphism.
 \end{proof}
\end{lemma}

It is known the Picard groups $Pic(M_{_i})$ (resp. $Pic(M_{_i}')$), $i=1,2$, are isomorphic to $\Z$. Let $\theta_{_i}$ (resp. $\theta_{_i}'$), $i=1,2$, be the unique ample generators of $Pic(M_{_i})$ (resp. $Pic(M_{_i}')$). Let us denote the projections $M_{_1}\times M_{_2}\to M_{_i}$ by $s_{_i}$; the projections $M_{_1}'\times M_{_2}'\to M_{_i}'$ by $s_{_i}'$ and the projections $P_{_1}\times P_{_2}\to P_{_i}$ by $r_{_i}$, $i=1,2$. Then we immediately get the following relations: 
 $$
r_{_1}^*q_{_1}^*\theta_{_1}=q^*s_{_1}^*\theta_{_1}, \hspace{0.5cm} r_{_1}^*q_{_1}'^*\theta_{_1}'=q'^*s_{_1}'^*\theta_{_1}', $$ $$r_{_2}^*q_{_2}^*\theta_{_2}=q^*s_{_2}^*\theta_{_2}, \hspace{0.5cm} r_{_2}^*q_{_2}'^*\theta_{_2}'=q'^*s_{_2}'^*\theta_{_2}'. $$

 Let $\Theta_{_1}:=q^*s_{_1}^*\theta_{_1}$, $\Theta_{_2}:=q'^*s_{_1}'^*\theta'_{_1}$, $\Theta_{_3}:=q^*s_{_2}^*\theta_{_2}$ and $\Theta_{_4}:=q'^*s_2'^*\theta'_{_2}$.
 \begin{lemma}\label{ind}
  The line bundles $\Theta_{_i}$, $i=1,\cdots 4$ are linearly independent on $P_{_1}\times P_{_2}$.
  \begin{proof}
     Note that $q_{_1}^*\theta_{_1}'$ and $q_{_1}'^*\theta_{_1}'$ are linearly independent line bundles over $P_{_1}$. This follows from the observation: The line bundles $\theta_{_1}$ (resp. $\theta_{_1}'$) is ample over $M_{_1}$ (resp. $M_{_1}'$) and $M_{_1}$, $M_{_1}'$ are normal varieties. Moreover, by Fact \ref{para} the fibre $q_{_1}^{-1}(F)$ is isomorphic to $\bp^1$ for all $F\in M_{_1}$ and the fibre $q_{_1}'^{-1}(F)$ is isomorphic to $\bp^1$ for all $F\in M_{_1}'^s$. The image of the morphism, inside some projective space, defined by the linear system corresponding to the sufficiently large power of $q_{_1}^*\theta_{_1}$ (resp. $q_{_1}'^*\theta_{_1}'$) will be isomorphic to $M_{_1}$ (resp. $M_{_1}'$) (this follows by Lemma \ref{good}). Since $M_{_1}$ and $M_{_1}'$ are not isomorphic we have $q_{_1}^*\theta_{_1}'$ and $q_{_1}'^*\theta_{_1}'$ are linearly independent. Now $r_{_i}:P_{_1}\times P_{_2}\to P_{_i}$ are the projection maps, $i=1,2$. Therefore, $\Theta_{_1}= r_{_1}^*q_{_1}^*\theta_{_1}$ and $\Theta_{_2}=r_{_1}^*q_{_1}'^*\theta_{_1}'$ are linearly independent. Similarly $\Theta_{_3}$ and $\Theta_{_4}$ are linearly independent. Next we will show that the relation $\Theta_{_1} ^{a_{_1}}\otimes  \Theta_{_2}^{a_{_2}}=\Theta_{_3} ^{a_{_3}}\otimes  \Theta_{_4}^{a_{_4}}$, with all $a_{_i}$ non zero will never occur. The above relation would imply $r_{_1}^*L=r_{_2}^*M$ where $L$ is a non trivial line bundle on $P_{_1}$ and $M$ is a nontrivial line bunlde on $P_{_2}$. But this is impossible as $r_{_1}^*L|_{_{q_{_1}^{-1}(x)}}$ is trivial but $r_{_2}^*M|_{_{q_{_1}^{-1}(x)\simeq P_{_2}}}\simeq M$ is non trivial for $x\in P_{_1}$. From the above observation we conclude that $\Theta_{_i}$ are linealy independent, $i=1,2,3,4$.
  \end{proof}
\end{lemma}
In section $5$ we will observe that $P_{_i}$'s are rational varieties and $Pic(P_{_i})\simeq \Z\oplus \Z$, $i=1,2$. Thus $Pic(P_{_1}\times P_{_2})\simeq Pic(P_{_1})\times Pic(P_{_2})\simeq \Z^4$. Therefore, the line bundles $\Theta_{_i}$ generate the picard group $Pic(P_{_1}\times P_{_2})$.
 Now we will prove the following:
 \begin{lemma}\label{key}
  The morphism $i^*-j^*:H^2(M_{_{12}})\oplus H^2( M_{_{21}})\to H^2(P_{_1}\times P_{_2})\simeq H^2(D)$, is surjective where $i$ and $j$ are the inclusions in the diagrams \ref{parab}, \ref{parab1}.
  \begin{proof}
   %We will find a set of generators $e_{_i}$, $i=1,\cdots 4$ for $H^2(P_{_1}\times P_{_2}^s)$ and show that there exists $f_{_i}\in H^2(P)\oplus H^2(\overline P)$ such that $(i^*-j^*)(f_{_i})=e_{_i}$. 
  Since $P_{_i}$ are rational varieties, we have $c_{1}:Pic(P_{_1}\times P_{_2})\to H^2(P_{_1}\times P_{_2},\Z)$ is an isomorphism where $c_{_1}$ is the first chern class homomorphism. By Lemma \ref{topology1} we get $Pic(P_{_1}\times P_{_2})\simeq H^2(P_{_1}\times P_{_2},\Z)\simeq \Z^{4}$   Therefore, by Lemma \ref{ind} $H^2(P_{_1}\times P_{_2})$ is generated by $c_{_1}(\Theta_{_1})$, $c_{_1}(\Theta_{_2})$, $c_{_1}(\Theta_{_3})$ and $c_{_1}(\Theta_{_4})$. From commutative diagrams \ref{parab} and \ref{parab1} we get \[\Theta_{_1}=i^*(p^*s_{_1}^*\theta_{_1}), ~ \Theta_{_2}=j^*(p'^*s_{_1}'^*\theta_{_1}'), ~ \Theta_{_3}=i^*(p^*s_{_2}^*\theta_{_2}) ~ \mbox{and } ~ \Theta_{_4}=j^*(p'^*s_{_2}'^*\theta_{_2}').\] Therefore, $H^2(M_{_{12}})\oplus H^2(M_{_{21}})\to H^2(P_{_1}\times P_{_2}^s)$ is surjective.
 \end{proof}
\end{lemma}
We will now prove the main theorem of this section.
\begin{theorem}
 $(i)$ $H^2(M_{_{0,\xi}})\simeq \Q^{2}$ and $(ii)$ $H^3(M_{_{0,\xi}})\simeq \Q^{2g}$.
 \begin{proof}
  By Lemma \ref{key}, \ref{third1} and using Mayer-Vietoris sequence we get-
  \[0\to H^2(M_{_{0,\xi}})\to H^2(M_{_{12}})\oplus H^2(M_{_{21}})\to H^2(D)\to 0.\]
  and \[0\to H^3(M_{_{0,\xi}})\to H^3(M_{_{12}})\oplus H^3(M_{_{21}})\to H^3(D)\to 0.\]
  
  Now $b_{_2}(M_{_{12}})=b_{_2}(M_{_{21}})=2$ by Theorem \ref{great} and $b_{_2}(D)=b_{_2}(P_{_1}\times P_{_2})=4$ by Lemma \ref{topology1}. Therefore, $b_{_2}(M_{_{0,\xi}})=2$.
  
 Also we have $b_{_3}(M_{_{12}})=b_{_3}(M_{_{21}})=2g$ by Theorem \ref{great} and $b_{_3}(D)=b_{_3}(P_{_1}\times P_{_2})=2g$ by Lemma \ref{topology1}. Therefore, $b_{_3}(M_{_{0,\xi}})=b_{_3}(M_{_{12}})+b_{_3}(M_{_{21}})-b_{_3}(D)=2g$. %As both $H^2(M_{_{0,\xi}},\Z)$ and $H^3(M_{_{0,\xi}},\Z)$ are torsion free, we conclude that $H^2(M_{_{0,\xi}},\Z)\simeq \Z^2$ and $H^3(M_{_{0,\xi}},\Z)\simeq \Z^{2g}$.
 \end{proof}
\end{theorem}
\subsection{Hodge structure on \texorpdfstring{$H^3(M_{_0},\Z)$}{}}
\begin{theorem}\label{pure}
 The Hodge structure on $H^3(M_{_0},\Z)$ is pure of weight $3$ with $h^{^{3,0}}=h^{^{0,3}}=0$.
\end{theorem}
 \begin{proof}
 We have the following short exact sequence: \[0\to H^3(M_{_{0,\xi}},\C)\stackrel{r^*}{\to} H^3(M_{_{12}},\C)\oplus H^3(M_{_{21}},\C)\stackrel{i^*-j^*}{\to} H^3(D,\C)\to 0.\] where all the morphisms are the morphism of Hodge structures. Thus $Ker(i^*-j^*)$ is a pure sub Hodge structure of $H^3(M_{_{12}},\C)\oplus H^3(M_{_{21}},\C)$ of weight $3$. This induce a Hodge structure of weight $3$ on $H^3(M_{_{0,\xi}},\C)$ as $H^3(M_{_{0,\xi}},\C)$ is isomorphic to $Ker(i^*-j^*)$.
 Since  $M_{_{12}}$ and $M_{_{21}}$ are smooth unirational varieties (see Lemma \ref{unirational}) and their intersection $D=M_{_{12}}\cap M_{_{21}}$ is also a smooth unirational variety, we have $h^{^{3,0}}(M_{_{12}})=h^{^{3,0}}(M_{_{21}})=0$ and $h^{^{3,0}}(D)=0$. Hence we conclude $h^{^{3,0}}(Ker(i^*-j^*))=h^{^{3,0}}(H^3(M_{_{0,\xi}},\Z))=0$. This completes the proof.
   \end{proof}
 \begin{remark}
  Thus we can define the intermediate Jacobian as in \ref{inter} corresponding to the Hodge structure on $H^3(M_{_{0,\xi}},\Z)$. We will denote this intermediate Jacobian by $J^2(M_{_0})$.
 \end{remark}

\section{Degeneration of the intermediate Jacobian of the moduli space}\label{degen1}
Let $\pi:\mathcal{X}\to C$ be a proper, flat and surjective family of curves, parametrised by a smooth, irreducible curve $C$. We assume that $\mathcal{X}$ is a smooth variety over $\C$. Fix a point $0\in C$. We assume that $\pi$ is smooth outside the point $0$ and $\pi^{-1}(0)=X_{_0}$ where $X_{_0}$ is a reducible curve with two smooth, irreducible components meeting at a node. Fix a line bundle $\mathcal{L}$ over $\mathcal{X}$ such that the restriction $\mathcal{L}_{_t}$ to $X_{_t}$ is a line bundle with Euler characteristic $\chi-(1-g)$ for all $t\in C$ where $g$ is the genus of $X_{_t}$. We denote the restriction $\mathcal{L}|_{_{X_{_0}}}$ by $\xi$. In Appendix we will show that  there is a proper, flat, surjective family $\pi':\mathcal{M}_{_{\mathcal{L}}}\to C$ such that $\pi'^{-1}(0)=M_{_{0,\xi}}$, the moduli space of rank $2$, stable torsion free sheaves with determinant $\xi$ and for $t\neq 0$, $\pi'^{-1}(t)=M_{_{t,\mathcal{L}_t}}$, the moduli space of rank $2$ stable bundles on $X_{_t}$ with determinant $L_{_t}$ (see  Proposition \ref{reldet} and Remark \ref{reldet1} in the Appendix). Moreover, $\mathcal{M}_{_{\mathcal{L}}}$ is smooth over $\C$. Choose a neighbourhood of the point $0$ which is analytically isomorphic to the open unit disk $\Delta$ such that both the morphisms $\pi'|_{_{\Delta^*}}$ and $\pi|_{_{\Delta^*}}$ are smooth, $\Delta^*:=\Delta-0$. Denote the family $\pi'|_{_{\Delta^*}}:\pi'^{-1}(\Delta^*)\to \Delta^*$ by $\{M_{_t}\}_{_{t\in \Delta^*}}$ and the family $\pi|_{_{\Delta^*}}:\pi^{-1}(\Delta^*)\to \Delta^*$ by $\{X_{_t}\}_{_{t\in \Delta^*}}$ .

\subsubsection{\bf {Variation of Hodge structure corresponding to the family \texorpdfstring{$\{M_{_t}\}_{_{t\in \Delta^*}}$}{} and \texorpdfstring{$\{X_{_t}\}_{_{t\in \Delta^*}}$}{}}} 
 Since $\pi'|_{_{\Delta^*}}$ (resp.$\pi|_{_{\Delta^*}}$) is smooth we get a local system $R^i\pi'_{_*}\Z$ (resp. $R^i\pi_{_*}\Z$) for $i\geq 0$, of free abelian groups whose fibre over a point $t\in \Delta^*$ is isomorphic to $H^i(M_{_t},\Z)$ (resp. $H^i(X_{_t},\Z)$). Let $H_{_{\Z}}(\mathcal{M}):=R^3\pi_{_*}'\Z$ and $H_{_{\Z}}(\mathcal{X}):=R^1\pi_{_*}\Z$. Let $H_{_{\C}}(\mathcal{M})$ (resp. $H_{_{\C}}(\mathcal{X})$) be the holomorphic bundle over $\Delta^*$ whose sheaf of section is $H_{_{\Z}}(\mathcal{M})\otimes_{\Z}\mathcal{O}_{_{\Delta^*}}$ (resp. $H_{_{\Z}}(\mathcal{X})\otimes_{\Z}\mathcal{O}_{_{\Delta^*}}$). Then $H_{_{\C}}(\mathcal{M})$(resp. $H_{_{\C}}(\mathcal{X})$) admits a flat connection $\nabla_{_{\mathcal{M}}}$ (resp. $\nabla_{_{\mathcal{X}}}$). Let $T$ (resp. $T'$) be the monodromy operator defined by the flat connection $\nabla_{_{\mathcal{M}}}$ (resp. $\nabla_{_{\mathcal{X}}}$) corresponding to the positive generator of $\pi_{_1}(\Delta^*,t_{_0})$. Since the fibre $\pi'^{-1}(0)$ (resp. $\pi^{-1}(0)$) is union of two smooth projective varities intersecting transversally, $T$ (resp. $T'$) is unipotent. It is known that the unipotency index of $T$ is atmost $4$ and $T'$ is atmost $2$ (see \cite[Monodromy Theorem page 106]{M}). Recall that there is a decreasing filtration $\{F^p\}$,$p=0,1,2,3$  (resp. $\{G^q\}$, $q=0,1$) of the holomorphic vector bundle $H_{_{\C}}(\mathcal{M})$ (resp. $H_{_{\C}}(\mathcal{X})$) by holomorphic subbundles such that 
\begin{enumerate}
 \item for each $t\in \Delta^*$ the filtration $\{F^p(t)\}$ (resp. $\{G^p(t)\}$) gives the Hodge filtration on $H_{_{\C}}(\mathcal{M})(t)=H^3(M_{_t},\C)$.
 \item the flat connection $\nabla_{_{\mathcal{M}}}$ (resp. $\nabla_{_{\mathcal{X}}}$ ) satisfies $\nabla_{_{\mathcal{M}}}(F^p)\subset F^{p-1}\otimes {\Omega^1_{_{\Delta^*}}}$ (resp. $\nabla_{_{\mathcal{X}}}(G^q)\subset G^{q-1}\otimes {\Omega^1_{_{\Delta^*}}}$) where ${\Omega^1_{_{\Delta^*}}}$ is the sheaf of holomorphic $1$ forms on $\Delta^*$.
\end{enumerate}

We say the triple $(H_{_\C}(\mathcal{M}),F^p,H_{_{\Z}}(\mathcal{M}))$ (resp. $(H_{_\C}(\mathcal{M}),G^q,H_{_{\Z}}(\mathcal{M}))$ a variation of Hodge structure corresponding to the local system $H_{_{\Z}}(\mathcal{M})$ (resp.$H_{_{\Z}}(\mathcal{X})$). Set $V=H_{_{\C}}(\mathcal{M})/F^2$. Then $V$ is a holomorhic vector bundle over $\Delta^*$ of rank $2g$. Since $H_{_{\Z}}(\mathcal{M})(t)\cap F^2(t)=(0)$ for all $t\in \Delta^*$, we get $H_{_{\Z}}(\mathcal{M})$ is a cocompact lattice inside $V$ i.e for each $t\in \Delta^*$ $H_{_{\Z}}(\mathcal{M})(t)\subset V(t)$ is a cocompact lattice of rank $2g$. Therefore, we can construct a complex manifold $J^{2*}:=V/H_{_{\Z}}(\mathcal{M})$ and a proper, surjective, holomorphic submersion $\pi_{_1}:J^2\to \Delta^*$ such that $\pi_{_1}^{-1}(t)=J^2(M_{_t})$ for all $t\in \Delta^*$. Similarly, we can construct a proper, holomorphic, submerssion $\pi':J^{0*}\to \Delta^*$ such that $\pi'^{-1}(t)=J^0(X_{_t})$. Note that the principal polarisations $\{\Theta'_{_t}\}_{_{t\in \Delta^*}}$, induced by the intersection form $(1.0.2)$, fit together to give a relative polarisation $\Theta'$ on $J^{2*}$. Also the family of Jacobians $\pi':J^{0*}\to \Delta^*$ carries a canonical relative principal polarisation induced by the intersection pairing on $H^1(X_{_t},\Z)$, $t\in \Delta^*$. We denote this relative polarisation by $\Theta$.

We recall that there is a  unique extension $\overline{H}_{_{\C}}(\mathcal{M})$ (resp. $\overline{H}_{_{\C}}(\mathcal{X})$) of the holomorphic vector bundle $H_{_{\C}}(\mathcal{M})$ such that the extended connection $\overline{\nabla}_{_{\mathcal{M}}}$ (resp. $\overline{\nabla}_{_{\mathcal{X}}}$) is regular singular and the residue $N=log(T)$ (resp. $N'=log(T')$) of $\overline{\nabla}_{_{\mathcal{M}}}$ (resp. $\overline{\nabla}_{_{\mathcal{X}}}$) is nilpotent ((see \cite[Page 91-92]{de}). 
There is also an extension $\overline{F^2}$ (resp. $\overline{G^1}$) of the subbundle $F^2$ (resp. $G^1$) such that $\overline{\nabla}_{_{\mathcal{M}}}(\overline{F^2})\subset \overline{F^1} \otimes \Omega_{\Delta}^1(log(0))$ (\cite{S}, Nilpotent orbit Theorem). Let $\overline{H}_{_{\Z}}(\mathcal{M}):=j_{_*}(H_{_{\Z}}(\mathcal{M}))$ (resp. $\overline{H}_{_{\Z}}(\mathcal{X}):=j_{_*}(H_{_{\Z}}(\mathcal{X}))$ where $j:\Delta^*\to \Delta$ is an inclusion. We denote by $\overline{H}(\mathcal{M})(0)$ (resp. $\overline{H}(\mathcal{X})(0)$) and $\overline{F}^2(0)$ ($\overline{G}^1(0)$) the fibre of $\overline{H}_{_{\C}}(\mathcal{M})$ (resp. $\overline{H}_{_{\C}}(\mathcal{X})$) and $\overline{F}^2$ (resp. $\overline{G}^1$) at $0$.

%Now we consider the family $\pi:\mathcal{X}\to \Delta$ where $\pi^{-1}(t)$ is a smooth projective curve of genus $g$ if $t\neq 0$ and $\pi^{-1}(0)=X_{_0}$. Let $H_{_{\Z}}:=R^1\pi_{_*}\Z$ and $H_{_{\C}}$ be the holomorphic bundle on $\Delta^*$ whose sheaf of sections is $H_{_{\Z}}\otimes_{\Z}\mathcal{O}_{_{\Delta^*}}$. Let $(H_{_{\C}},G^q,H_{_{\Z}})$ be the variation of Hodge structure corresponding to the local system $H_{_{\Z}}$. Hence, there is a holomorphic family $\pi_{_2}:J^{0*}\to \Delta^*$ such that $\pi_{_2}^{-1}(t)=J(X_{_t})$, $t\in \Delta^*$. Let $(\overline{H}_{_{\C}},\overline{F^q},\overline{H}_{_{\Z}})$,$q=0,1$, be canonical extension as described in the previous paragraph.
%We denote by $\overline{H}(0)$ and $\overline{F}^1(0)$ the fibre of $\overline{H}_{_{\C}}$ and $\overline{F}^1$ at $0$. 
\subsubsection{\bf{Limiting mixed Hodge structure on the fibre \texorpdfstring{$\overline{H}(\mathcal{M})(0)$}{} and \texorpdfstring{$\overline{H}(\mathcal{X})(0)$}{}}}\label{clem}
In general the filtration $\{\overline{F^p}(0)\}$ (resp. $\{\overline{G^q}(0)\}$) does not define a Hodge structure on $\overline{H}(\mathcal{M})(0)$ (resp. $\overline{H}(\mathcal{X})(0)$). %Therefore, in general we have $\overline{H}_{_{\Z}}(\mathcal{M})(0)\cap \overline{F^2}(0)\neq (0)$ (resp. $\overline{H}_{_{\Z}}(\mathcal{X})(0)\cap \overline{G^1}(0)\neq (0)$). 
However, it follows that from a theorem of W Schmid \cite{S} (see \cite[Theorem 10]{RH} for the statement), for each $t\in \C^*$, the data $(t^N\overline{H}_{_{\Z}}(\mathcal{M})(0),\overline{F^p}(0), W_{_r})$ (resp. $(t^{N'}\overline{H}_{_{\Z}}(\mathcal{X})(0),\overline{G^q}(0), W'_{_r})$) defines a mixed Hodge structure (see \cite[definition 11]{RH} for the definition) where  $N$ (resp. $N'$) is the residue of the monodromy operator $T_{_{\mathcal{M}}}$ (resp.$T_{_{\mathcal{M}}}$) and $W_{_r}$ (resp. $W'_{_r}$) is the weight filtration defined by $N$ (resp. $N'$)  (see also \cite{M}). Thus, in particular, if the residue of the monodromy operator is trivial then there is no monodromy weight filtration and we have pure Hodge structure.

In the next lemma we will show that $N'=0$. As a consequence of this we can extend the family $J^{0*}\to \Delta^*$ to a family $J^0\to \Delta$.
\begin{lemma}\label{lim}
 The limiting Hodge structure on $\overline{H}(\mathcal{X})(0)$ is pure and is isomorphic to the Hodge structure on $H^1(X_1,\C)\oplus H^1(X_2,\C)$.
\end{lemma}
 \begin{proof}
  Since the singular fiber $X_{_0}$ is the union of two smooth curves meeting transversally at a node we have $N'=0$ (see \cite[page 111]{M}). So, in this case, there is no weight filtration and hence the limiting Hodge structure on $\overline{H}(\mathcal{X})(0)$ is pure. Now we have a morphism of MHS, $i^*:H^1(X_{_0},\Z)\to \overline{H}(\mathcal{X})(0)$ of $(0,0)$ type (see \cite[Clemens-Schmid I,page 108]{M}). By Local Invariance Cycle Theorem \cite[page 108]{M}), it  known that: \[Ker(N)=Im(i^*).\] Since $Ker(N')=\overline{H}(0)$, $i^*$ is surjecive. Now $rk(H^1(X_{_0},\Z))=2g=rk(\overline{H}(\mathcal{X})(0)))$. Therefore, $i^*:H^1(X_{_0},\Z)\to \overline{H(\mathcal{X})}(0)$ is an isomorphism of Hodge structure.
(see \cite[page 111]{M}).
  \end{proof}
\begin{corollary}
 There is a holomorphic family $\pi_{_2}:J^0\to \Delta$ extending the family $\pi_{_2}:J^{0*}\to \Delta^*$ such that $\pi_{_2}^{-1}(0)=J^0(X_{_0})$.
\end{corollary}
 \begin{proof}
 
Since $N'=0$, we get that $\overline{G^1}(0)\cap \overline{H}_{_{\Z}}(\mathcal{X})(0)=(0)$. As a consequence $\overline{H}_{_{\Z}}(\mathcal{X})(0)$ is a full lattice inside $\overline{H}_{_{\C}}(\mathcal{X})(0)/\bar{G}^1(0)$. Thus there is a holomorphic family $\pi_{_2}:J^0(\mathcal{X})\to \Delta$ extending the family $\pi_{_1}:J^{0*}\to \Delta^*$ such that $\pi_{_2}^{-1}(0)=V/\overline{H}_{_{\Z}}(0)$ where $V:=\overline{H}_{_{\C}}(0)/\overline{F^1}(0)$. By Lemma \ref{lim}, it follows that $\pi_{_2}^{-1}(0)\simeq J^0(X_{_0})$.  
 \end{proof}

Next we shall show,
\begin{lemma}\label{Hodge}
There is an isomorphism $\overline{\phi}:\overline{H_{_{\mathbb C}}}(\mathcal{X})\to \overline{H_{_{\mathbb C}}}(\mathcal{M})$ such that $\overline{\phi}(\overline{G^q})= \overline{F^{q+1}}$ and $\overline{\phi}(\overline{H_{_{\mathbb Z}}}(\mathcal{X}))=\overline{H_{_{\mathbb Z}}}(\mathcal{M})$, $q=0,1$.
\end{lemma}
\begin{proof}
Let  $\mathcal{U}$ be the relative universal bundle over $\mathcal{X}^*\times_{_{\Delta^*}} \mathcal{M}^*$ i.e $\mathcal{U}_{_{_{|X_{_t}\times M_{_t}}}}$ is the corresponding universal bundle. Now if we consider  $(1,3)$ Kunneth-component $[c_{_2}(\mathcal{U})|_{_{X_{_t}\times M_{_t}}}]_{_{1,3}}  \in H^1(X_{_t},\Z)\otimes H^3(M_{_t},\Z)$ of $c_{_2}(\mathcal{U}|_{_{X_{_t}\times M_{_t}}})$, then we get a morphism $\phi_t:H^1(X_{_t},\Z)\to H^3(M_{_t},\Z)$, $t\in \Delta^*$ such that $\phi_t(G^q(t))\subseteq F^{q+1}(t)$ for $q=0,1$ (see \cite{mum-new}). Thus we get a morphism $\phi:H_{_{\Z}}(\mathcal{X})\to H_{_{\Z}}(\mathcal{M})$ of local systems, preserving the Hodge filtrations. Since the filtration $\{\overline{G^q}(0)\}$ (resp. $\{\overline{F^p}(0)\}$) is canonically determined by the filtrations $\{G^q(t)\}$ (resp. $\{F^q(t)\}$)(see \cite[Theorem(Schmid),page 116]{M}), we get a morphism $\phi_0:\overline{H}(\mathcal{X})(0)\to \overline{H}(\mathcal{M})(0)$ such that $\phi_0(\overline{G^q}(0))\subseteq \overline{F^{q+1}}(0)$. Therefore, the morphism $\phi$ extends to a morphism $\overline{\phi}:\overline{H_{_{\C}}}(\mathcal{X})\to \overline{{H}_{_{\C}}}(\mathcal{M})$ such that $\overline{\phi}(\bar{G^q})\subseteq \overline {F^{q+1}}$; further, we have $\phi^*(\Theta')=\Theta$ (see  \cite[Section 5,page 625]{
Ba2}). 
By the Mumford-Newstead theorem \cite[Proposition 1, page 1204]{mum-new} we conclude that $\overline{\phi}$ is an isomorphism.
\end{proof}

 Now we will state the main theorem of this section:
 \begin{theorem}\label{main}{\ }
 
 \begin{enumerate}
  \item There is a holomorphic family  $\{J^2(M_{t,\mathcal{L}_{t}})\}_{_{t\in \Delta}}$ of intermediate Jacobians corresponding to the family $\{M_{_{t, \mathcal L_t}}\}_{_{t\in \Delta}}$. In other words, there is a surjective,
   proper, holomorphic submersion 
    \[\pi_{_2} : J^2(\mathcal{M}_{_{\mathcal{L}}})\longrightarrow \Delta\] 
   such that $\pi_2^{-1}(t)=J^2(M_{_{t, \mathcal L_t}})~ \forall ~  t\in \Delta^{*}:=\Delta \setminus \{0\}$ and $\pi_2^{-1}(0)=J^2(M_{_{0,\xi}})$.   
   Further, we show that there exists a relative ample class $\Theta'$ on $J^2(\mathcal{M}_{_{\mathcal{L}}} )_{|\Delta^*}$ such that $\Theta'_{|J^2(M_{_{t,\mathcal L_t}})}=\theta'_{_t}$, where $\Theta'_{_t}$ is the principal polarisation on $J^2(M_{_{t, \mathcal L_t}})$.
   
  \item 
   There is an isomorphism 
   \begin{equation}
\xymatrix{
J^0(\mathcal{X}) \ar[rr]^{\Phi}_{\sim} \ar[rd]_{\pi_{_1}} && J^2(\mathcal{M}_{_{\mathcal{L}}})\ar[ld]^{\pi_{_2}} \\
& \Delta
}
\end{equation} 
such that $\Phi^*\Theta'_{|\pi_1^{-1}(t)}=\Theta_{_t}$ for all $t\in \Delta^*$, where $\pi_1: J^0(\mathcal X)\to \Delta$ is the holomorphic family $\{J^0(X_{_t})\}_{_{t\in \Delta}}$ of Jacobians and $\Theta_{_t}$ is  the canonical polarisation on $J^0(X_{_t})$. In particular, $J^2(\mathcal{M}_{_{\mathcal{L}}})_{_0}:=\pi_{_2}^{-1}(0)$ is an abelian variety.
\end{enumerate}
\begin{proof}
 Proof of $(1)$: By Lemma \ref{Hodge} we get the local system $H_{_{\mathbb{Z}}}(\mathcal{M})$ is isomorphic to the local system $H_{_{\mathbb{Z}}}(\mathcal{X})$ over $\Delta^*$. Since by Lemma \ref{lim} the local system $H_{_{\mathbb{Z}}}(\mathcal{X})$ has trivial monodromy, therefore the local system $H_{_{\mathbb{Z}}}(\mathcal{M})$ also has trivial monodromy. Hence $N=0$. Thus we have $\overline{F^2}(0)\cap \overline{H}_{_{\mathbb{Z}}}(\mathcal{M})(0)=(0)$. As a consequence we have a holomorphic family $\pi_{_1}:J^2(\mathcal{M}_{_{\mathcal{L}}})\to \Delta$ extending the family $\pi_{_1}:J^{2*}\to \Delta^*$ such that $\pi^{-1}(0)=V'/\overline{H}_{_{\mathbb{Z}}}(\mathcal{M})(0)$ where $V'=\overline{H}_{_{\mathbb{C}}}(\mathcal{M})(0)/\overline{F}^2(0)$.   
 Now we claim: $\pi_{_1}^{-1}(0)\simeq J^2(M_{_{0,\xi}})$. By Theorem \ref{pure}, we see that the Hodge structure on $H^3(M_{_{0,\xi}},\mathbb{Z})$ is pure and it has rank $2g$. Now there is a morphism $i^*: H^3(M_{_{0,\xi}},\mathbb{Z})\to \overline{H}_{_{\mathcal{M}}}(0)$ of MHS of $(0,0)$ type and  $Ker(N)=Im(i^*)$. Since $Im(N)=0$ and both the Hodge structures have the same rank $2g$,  $\overline{H}_{_{\mathbb{Z}}}(\mathcal{M})(0)$ and $H^3(M_{_{0,\xi}},\mathbb{Z})$ are isomorphic as Hodge structures. This completes the proof of $(1)$.
 
 Proof of $(2)$: This immediately follows from Lemma \ref{Hodge}.
\end{proof}

\end{theorem}
As a corollary of the theorem we get the following result:
\begin{corollary}\label{MN} 
  Let $X_{_0}$ be a  projective curve with exactly two smooth irreducible components $X_1$ and $X_2$ meeting at a simple node $p$. We further assume that $g_{_i}>3$, $i=1,2$. Then,  
  there is an isomorphism $J^0(X_{_0})\simeq J^2(M_{_{0,\xi}})$, where $\xi\in J^{\chi}(X_{_0})$.
   In particular, $J^2(M_{_{0,\xi}})$ is an abelian variety.
  \end{corollary} 
 \begin{proof}
  By our genus assumption: $g_{_i}>3$ for i=1,2, the curve $X_{_0}$ is stable i.e, they have finite number of automorphisms. As the moduli space of stable curves is complete, we get an algebraic family $r:\mathcal{X}\to \bp^1$ such that $r_{_1}^{-1}(t)$ is smooth if $t\neq t_{_0}$ and $r_{_1}^{-1}(t_{_0})=X_{_0}$. Moreover, we can choose $\mathcal{X}$ to be regular over $\C$. Therefore, by  Theorem \ref{main}, we get $J^2(M_{_{0,\xi}})\simeq J(X_{_0})$. Hence, $J^2(M_{_{0,\xi}})$ is an abelian variety.
 \end{proof}
\section{ Torelli type Theorem for the moduli space of rank \texorpdfstring{$2$}{} degree \texorpdfstring{$1$}{} fixed determinant torsion free sheaves over a reducible curve }\label{torel}
In this section our goal is to investigate the moduli space $M_{_{0,\xi}}$ more carefully, and show that we can actually recover the curve $X_0$ i.e both the components as well as the node, from the moduli space $M_{_{0,\xi}}$ following  a strategy given in \cite{bbd}.

Let $\pi:\tilde X_0\to X_0$ be the normalization map and $\pi^{-1}(p)=\{x_{_1},x_{_2} \}$, where $p\in X_{_1}\cap X_{_2}$.  Note that $\tilde X_{_0}=X_{_1}\sqcup X_{_2}$, the disjoint union of $X_{_1}$ and $X_{_2}$. Fix a line bundle $\xi$ on $X_{_0}$ and let $\xi_{_i}=\xi|_{X_{_i}}$, $i=1,2$. Recall that that the moduli space $M_{_{0,\xi}}$ of rank $2$ stable torsion free sheaves with determinant $\xi$ over $X_0$ is the union of two irreducible, smooth, projective varieties intersecting transversally along a divisor $D$.
 We have also observed that $D$ is isomorphic to
the product $P_1\times P_2$, where $P_{_1}$ is the moduli space of rank $2$ parabolic semistable bundles $(F_{_1},0\subset F^2F_{_1}(x_{_1})\subset F_{_1}(x_{_1}))$ over $X_1$ with $det\simeq \xi_{_1}$ and  weights $(\frac{a_{_1}}{2},\frac{a_{_2}}{2})$, and $P_{_2}$ is the moduli space of rank $2$ parabolic semistable bundles $(F_{_2},0\subset F^2F_{_2}(x_{_2})\subset F_{_2}(x_{_2}))$ over $X_2$ with $det\simeq \xi_{_2}$ and weights $(\frac{a_{_1}}{2},\frac{a_{_2}}{2})$, where $a=(a_{_1},a_{_2})$ is the polarisation on $X_{_0}$. Without loss of generality, we can assume that $deg(\xi_{_1})=1$ and $deg(\xi_{_2}=0$.

Let  $M_{_1}$ (resp. $ M'_{_1}$) be the moduli space of rank $2$, deg $1$,
semistable bundles over $X_1$ with det $E\simeq \xi_{_1}$ (resp. moduli space of rank $2$, deg $0$ semistable bundles over $X_{_2}$ with  $detE\simeq \xi_{_1}(-x_{_1})$).

Note that $Pic(M_{_1})\simeq \Z$ (resp. $Pic(M'_{_1})\simeq \Z$).
Let  $\theta_{_1} $ (resp. $\theta_{_1}' $) be the unique ample generator of $Pic(M_{_1})$ (resp. of $Pic(M'_{_1})$). It is known that there exists a unique 
rank $2$ bundle $\mathcal{E}$ over $X_{_1}\times M_{_1}$ such that $\wedge^2\mathcal{E}_{_{x_{_1}}}\simeq \theta_{_1}$, where $\mathcal{E}_{_{x_{_1}}}:=\mathcal{E}_{_{|x_{_1}\times M_{_1}}}$ (see \cite[Definition 2.10]{R}).
Since the weights $0<\frac{a_{_1}}{2},\frac{a_{_2}}{2}<1$ are very small, we can show that: $P_{_1}\simeq \bp(\mathcal{E}_{_{x_{_1}}})$ (see \cite[Proposition 6]{Ba1}). Therefore, it follows that $Pic(P_{_1})\simeq Pic(M_{_1})\oplus Pic(\bp^1) \simeq \Z\oplus \Z$.

We define a morphism $\pi_1': \bp(\mathcal{E}_{_{x_{_1}}})\to M'_{_1}$ as follows: Any closed point of $\bp(\mathcal{E}_{_{x_{_1}}})$ over $E\in M_{_1}$ looks like $\{E , V(x_{_1})\}$, where  $V(x_{_1})$ is a one dimensional subspace of the fibre $E(x_{_1})$. Consider the vector bundle $V$ 
which fits into the following exact sequence 
  \begin{equation}0\to V\to E\to (i_{_{x_{_1}}})_{_*}(E(x_{_1})/V(x_{_1}))\to 0.\end{equation}
  As $E(x_{_1})/V(x_{_1})$ is 1-dimensional vector space supported over the point $x_{_1}$, it follows that $det(V)\simeq \xi_{_1}(-x_{_1})$. We can easily check that $V$ is semistable (see \cite[page 11]{Ba1}). 
  
  Thus we get a Hecke correspondence: \begin{equation}
  \xymatrix{\bp(\mathcal{E}_{_{x_{_1}}}) \ar[r]^{\pi_1'} \ar[d]^{\pi_1} & M'_{_1} \\
  M_{_1}}
 \end{equation}
 
 Similarly, let $M_{_2}$ (resp. $M'_{_2}$) be the moduli space of rank $2$, deg $1$ semistable bundles over $X_2$ with $detE\simeq \xi_{_2}(x_{_2})$ (resp. the moduli space of rank $2$, deg $0$ semistable bundles over $X_2$ with $detE\simeq \xi_{_2}$).
 
 Let $\theta_{_2}$ (resp. $\theta_{_2}'$) be the unique ample generator of $Pic(M_{_2})$ (resp. $Pic(M'_{_2})$). Then there is a unique universal bundle  $\mathcal{E}'$ over $X_2\times M_{_2}$ such that  $\wedge^2 \mathcal{E}'_{_{x_{_2}}}\simeq \theta_{_2}$ where $\mathcal{E}'_{_{x_{_2}}}:=\mathcal{E}'_{_{|x_{_2}\times M_{_2}}}$.

Again, for the choice of weights $0<\frac{a_{_1}}{2},\frac{a_{_2}}{2}<1$, we have $P_2\simeq \bp(\mathcal{E}'_{_{x_{2}}})$ and we have a  Hecke correspondence as in the previous case:
\begin{equation}\label{eq:hecke}
  \xymatrix{\bp(\mathcal{E}'_{_{x_{_2}}}) \ar[r]^{\pi_2'} \ar[d]^{\pi_2} & M'_{_2} \\
  M_{_2}}
 \end{equation}  
 So, we have the following:
  \begin{equation}
  \xymatrix{\bp(\mathcal{E}_{_{x_{_1}}})\times \bp(\mathcal{E}'_{_{x_{_2}}}) \ar[r]^{p_{_1}}  & \bp(\mathcal{E}_{_{x_{_1}}}) \ar[d]^{\pi_1} \ar[r]^{\pi_1'}   & M'_{_1} \\
               & M_{_1}      }
 \end{equation}
  and 
  \begin{equation}
  \xymatrix{\bp(\mathcal{E}_{_{x_{_1}}})\times \bp(\mathcal{E}'_{_{x_{_2}}}) \ar[r]^{p_{_2}}  & \bp(\mathcal{E}'_{_{x_{_2}}}) \ar[d]^{\pi_2} \ar[r]^{\pi_2'}   & M'_{_2} \\
               & M_{_2}      }
 \end{equation}

 \begin{remark}\label{rational}
  Note that $M_{_1}$ and $M_{_2}$ are smooth ,projecive, rational varieties. Now $P_{_1}$ (resp. $P_{_2}$) is isomorphic to the projective bundle $\bp(\mathcal{E}_{_{x_{_1}}})$ (resp. $\bp(\mathcal{E}'_{_{x_{_2}}})$). Therefore, $P_{_i}$'s are rational varieties, $i=1,2$.
 \end{remark}

For the rest of the section we will fix the following notations:
$$
\vartheta_{_1}:=(\pi_1 \circ p_{_1})^*\theta_{_1}, \hspace{0.5cm} \vartheta_{_2}:=(\pi_1' \circ p_{_1})^*\theta_{_1}', $$ $$\vartheta_{_3}:=(\pi_2 \circ p_{_2})^*\theta_{_2}, \hspace{0.5cm} \vartheta_{_4}:=(\pi_2' \circ p_{_2})^*\theta_{_2}'. $$
\begin{proposition}\label{nef}
 The numerically effective cone of $P_{_1}\times P_{_2}$ is generated by the line bundles $ \vartheta_{_i} $, $i=1,2,3,4$.
\end{proposition}
 \begin{proof}
  Clearly, $ \vartheta_{_i} $, $i=1,\cdots, 4$ are numerically effective (nef) line bundles as they are the pull backs of the ample line bundles.  First we show that $ \vartheta_{_i} $, $i=1,\cdots, 4$, are linearly independent. Note that $\pi_{_1}^*\theta_{_1}$ and $\pi_{_1}'^*\theta_{_1}'$ are linearly independent over $\Z$ (see the proof of \cite[page 4, Theorem 2.1]{bbd} for an argument). Therefore, $\vartheta_{_1}=p_{_1}^*\pi_{_1}^*\theta_{_1}$ and $\vartheta_{_2}:=p_{_1}^{*}\pi'^*\theta'_{_1}$ are linearly independent. By similar reason $\vartheta_{_3}$ and $\vartheta_{_4}$ are linearly independnt. Now we show that the relation $\vartheta_{_1} ^{a_{_1}}\otimes  \vartheta_{_2}^{a_{_2}}=\vartheta_{_3} ^{a_{_3}}\otimes  \vartheta_{_4}^{a_{_4}}$ for some $a_{_i}\neq 0$,$i=1,\cdots 4$ will not occur. Suppose, $\vartheta_{_1} ^{a_{_1}}\otimes  \vartheta_{_2}^{a_{_2}}=\vartheta_{_3} ^{a_{_3}}\otimes  \vartheta_{_4}^{a_{_4}}$. Then this would imply $p_{_1}^*({\pi_{_1}^*\theta_{_1}}^{a_{_1}}\otimes {\pi_{_1}'^*\theta'_{_1}}^{a_{_2}})=p_{_2}^*({\pi_{_2}^*\theta_{_2}}^{a_{_3}}\otimes {\pi_{_2}'^*\theta'_{_2}}^{a_{_4}})$. But this is impossible for the following reason: The line bundle $p_{_1}^*(\pi_{_1}^*\theta_{_1}\otimes \pi_{_1}'^*\theta'_{_1})$ is trivial on the fibres of $p_{_1}$. But as the fibres of $p_{_1}$ are $P_{_2}$ and $\pi_{_2}^*\theta_{_2}\otimes \pi_{_2}'^*\theta'_{_2}$ is a non trivial line bundle on $P_{_2}$ we get $p_{_2}^*(\pi_{_2}^*\theta_{_2}\otimes \pi_{_2}'^*\theta'_{_2})$ is non trivial on the fibres of $p_{_1}$.
  From the above observation, it follows that $\vartheta_{_i}$, $i=1,\cdots 4$ are linearly independent. Since $P_{_1}$ and $P_{_2}$ are both rational varieties we get $Pic(P_{_1}\times P_{_2})\simeq Pic(P_{_1})\times Pic(P_{_2})\simeq \Z^4$. Therefore, any nef line bundle on $P_{_1}\times P_{_2}$ is a non negative linear combination of 
  $ \vartheta_{_1} $, $ \vartheta_{_2}$, $ \vartheta_{_3}$, $\vartheta_{_4}$. 
   
Next we show that $\bigotimes\limits_{i=1}^4 \vartheta_{_i}^{a_{_i}}$ is ample if $a_{_i}>0$ for all $i=1,\ldots,4$. It is enough to show that $\bigotimes\limits_{i=1}^4 \vartheta_{_i}$ is ample. 
    We observe that $\pi_1^{*}\theta_{_1} \otimes \pi_1'^{*}\theta'_{_1} $ (resp. $\pi_2^{*}\theta_{_2}\otimes \pi_2'^{*}\theta'_{_2}$) is ample on $P_{_1}$(resp. on $P_2$)(see the proof of \cite[Theorem 2.1, page 4, 3rd paragraph]{bbd}). Therefore , $\bigotimes\limits_{i=1}^4 \vartheta_{_i}$ is ample on $P_{_1}\times P_{_2}$.
    
  Finally, we have to show  $\bigotimes\limits_{i=1}^4 \vartheta_{_i}^{a_{_i}}$ is not ample if $a_{_i}=0$ for some $i$. Now fix $j\in \{1,\ldots,4\}$ such that $a_{_j}=0$. Then $\bigotimes\limits_{\substack{i=1 \\ i\neq j}}^4 \vartheta_{i}$ is not ample as it is the pull back of an ample line bundle from $P_{_k}\times M_{l}$ or $P_{_k}\times M_{_k}'$ for $l,k\in \{1,2\}$, $k\neq l$. Next we observe that if $i\in \{1,2\}$ and $j\in \{3,4\}$, then $\vartheta_{_i}\otimes \vartheta_{_j}$ is not ample. Since in this case it is pull back of an ample line bundle from $M_{_k}\times M_{l}$ or $M_{k}'\times M_{_l}'$ for  $k,l\in \{1,2\}$, $k\neq l$. We have already observed that $\vartheta_{_1}\otimes \vartheta_{_2}$ and $\vartheta_{_3}\otimes \vartheta_{_4}$ is not ample. 
  
  So, from the above observations, we conclude the proposition.  
 \end{proof} 
 \begin{lemma}\label{good}
  Let $f:X\to Y$ be a projective morphism with Y, a normal variety. Suppose, each fibre of $f$ is a rational variety. Let $L$ be a line bundle on $Y$ then $H^0(X,f^*L)\simeq H^0(Y,L)$.
  \begin{proof}
   Since the fibres of $f$ are connected and $Y$ is normal we have $\mathcal{O}_{_Y}\simeq f_{_*}\mathcal{O}_{_{X}}$. Thus $L\simeq f_{_*}f^{*}L$. Since all the fibres of $f$ are rational we get $H^i(X_{_y},L_{_y})=H^i(X_{_y},\mathcal{O}_{_{X_{_y}}})=0$ for all $i>0$. Hence $H^0(X,f^*L)\simeq H^0(Y,f_{_*}f^{*}L)=H^0(Y,L)$
  \end{proof}

 \end{lemma}

\begin{remark}\label{ima}
 Note that $\pi_1^{*}\theta_1\otimes \pi_1'^{*}\theta_1'$ (resp. $\pi_2^{*}\theta_2\otimes \pi_2'^{*}\theta_2'$) is ample on $P_{_1}$ (resp. on $P_{_2}$) (see the proof of \ref{nef}). Now $\vartheta_1\otimes  \vartheta_2=p_{_1}^*(\pi_1^{*}\theta_1\otimes \pi_1'^{*}\theta_1')$ and $\vartheta_3\otimes  \vartheta_4=p_{_1}^*(\pi_1^{*}\theta_3\otimes \pi_1'^{*}\theta_4')$. Since $P_{_1}$ and $P_{_2}$ are both rational varaities, by Lemma \ref{good}, the image of the morphism $|(\vartheta_1 \otimes \vartheta_2)^n|: P_1\times P_2\to \bp^N$  is isomorphic to $P_{_1}$ for some $n\gg 0$. Similarly, the image of the morphism $(\vartheta_3\otimes  \vartheta_4)^m:P_1\times P_2\to \bp^M$  
  is isomorphic to $P_2$ for some $m\gg 0$.
 \end{remark}

\begin{lemma}\label{nefc}
 Let $\theta$ be a nef but not ample line bundle on $P_1\times P_2$ 
 (i.e, $\theta$ lies in the boundary of the nef cone of $P_{_1}\times P_{_2}$) and $\theta \neq \vartheta_1^a\otimes  \vartheta_2^b ~ \mbox{or} ~  \vartheta_3^c\otimes \vartheta_4^d$, where $a, b, c$ and $d$ are some positive integers. 
 Let $Z$ be the image of the morphism $P_1\times P_2\to \bp^{N'}$ induced by the linear system $|\theta^n|$ for some large $n$. 
 Then we have $dim(Z)\neq dim(P_{_i})$ for $i=1,2$.
\end{lemma}
\begin{proof}
 Assume that $\theta \neq  \vartheta_1\otimes  \vartheta_2 ~ or ~  \vartheta_3\otimes  \vartheta_4$. Then $\theta$ is either of the form $\bigotimes\limits_{i\neq j} \vartheta_{i}$ for  $i,j\in \{1,\ldots,4\}$ or $\vartheta_{_i}\otimes \vartheta_{_j}$ for $i\in \{1,2\}$ and $j\in \{3,4\}$ (see the proof of \ref{nef}). 
 
 Fix $j\in \{1,2,3,4\}$. If $\theta$ is of the form $\bigotimes\limits_{\substack{i=1 \\ i\neq j}}^4 \vartheta_{i}$,  then the image $Z$ of the morphism $|\theta^n|$, for sufficiently large $n$, is either isomorphic to $P_{_k}\times M_{_l}$ or $P_{_k}\times M_{_l}'$ for $k,l\in \{1,2\}$, $k\neq l$.
 
 If $\theta$ is of the form $\vartheta_{_i}\otimes \vartheta_{_j}$ then the image $Z$ of the morphism $|\theta^n|$ is either isomorphic to $M_{_k}\times M_{_l}$ or $M_{_k}\times M_{_l}'$,  for $k,l\in \{1,2\}$, $k\neq l$. In both the cases we see  $dim(Z)\neq dim(P_{_i})$ and hence we are done.
 \end{proof}
 Now we prove the main theorem of this section:

Let  $X_{_0}$ (resp. $Y_{_0}$) be a reducible curve with two components $X_{_1}$, $X_{_2}$ (resp. $Y_{_1}$, $Y_{_2}$) meeting transversally at a point $p$ (resp. $q$). Let $\pi_{_1}:\tilde X_{_0}\to X_{_0}$ (resp. $\pi_{_2}:\tilde Y_{_0}\to Y_{_0}$) be the normalisation map and $\pi_{_1}^{-1}(p)=\{x_{_1},x_{_2}\}$, $\pi_{_2}^{-1}(q)=\{y_{_1},y_{_2}\}$. We will make the following assumption on the components of $X_{_0}$ and $Y_{_0}$.

\begin{itemize}
\item $g(X_{_i})=g(Y_{_i})\geq 2$ for $i=1,2$.  
\item $X_{_1}\ncong X_{_2}$ (resp. $Y_{_1}\ncong Y_{_2}$). 
\end{itemize}
Fix $\xi_{_{X_0}}\in J^{\chi}(X_{_0})$ (resp. $\xi_{_{Y_0}}\in J^{\chi}(Y_{_0})$).
Let  $M_{_{0,\xi_{_{X_0}}}}$ (resp. $M_{_{0,\xi_{_{Y_0}}}}$) be the moduli space of rank $2$, $a=(a_{_1},a_{_2})$-stable torsion free sheaves with $detE\simeq \xi_{_{X_0}}$ (resp.$detE\simeq \xi_{_{Y_0}}$) on $X_{_0}$ (resp. on $Y_0$). Let $D\subset M_{_{0,\xi_{_{X_0}}}}$ (resp. $D'\subset M_{_{0,\xi_{_{Y_0}}}}$) be the singular locus of $M_{_{0,\xi_{_{X_0}}}}$ (resp. $M_{_{0,\xi_{_{Y_0}}}}$) and $P_{_i}$ (resp. $P_{_i}'$) be the parabolic moduli spaces, described before, with parabolic structure over $x_{_i}$ (resp. $y_{_i}$). Then $D\simeq P_{_1}\times P_{_2}$ and $D'\simeq P'_{_1}\times P'_{_2}$. Now we have the following Torelli type theorem. 
    \begin{theorem}
   If $M_{_{0,\xi_{_{X_0}}}}\simeq M_{_{0,\xi_{_{Y_0}}}}$ then we have $X_{_0}\simeq Y_{_0}$.
  \end{theorem}
\begin{proof}
 Let $\Psi:M_{_{0,\xi_{_{X_0}}}}\simeq M_{_{0,\xi_{_{Y_0}}}}$ be an isomorphism. Then $\Psi(D)=D'$ as $D$ is the singular locus of $M_{_{0,\xi}}$. Therefore, $\Psi$ induces an isomorphism $\Psi:P_{_1}\times P_{_2}\simeq P_{_1}'\times P_{_2}'$. Now if we can show that the above statement will imply $P_{_i}\simeq P'_{_\sigma(i)}$ for $i\in\{1,2\}$ and $\sigma$ is a permutation on $\{1,2\}$. Then by \cite[Theorem 2.1]{bbd}, we get an isomorphism $f_{_i}:X_{_i}\to Y_{_{\sigma(i)}}$ such that $f_{_i}(x_{_i})=y_{_\sigma(i)}$. Hence, we get $X_{_0}\simeq Y_{_0}$. We will show that if $\Psi:P_{_1}\times P_{_2}\simeq P_{_1}'\times P_{_2}'$, then $P_{_i}\simeq P'_{_{\sigma(i)}}$. 
 %Therefore $\Psi$ induces an isomorphism $\Psi:P_1\times P_2\to P_1'\times P_2'$.
 %Let  $L_1=\tilde \theta_{_1} \otimes \tilde \theta_{_0} $, $L_2=\tilde \theta_{_1} '\otimes \tilde \theta_{_0}'$. We have already observed that  the image of the morphism defined by the linear system $|L_1^{n}|:P_1\times P_2\to \bp^N$, for some large $n$, is isomorphic to $P_{_1}$ and the image of the morphism defined by the linear system $|L_2^{m}|:P_1\times P_2\to \bp^M$, for some large $m$, is isomorphic to $P_2$.
 
 Let $\varsigma_1$, $\varsigma_2$, $\varsigma_3$, $\varsigma_4$ be the generators of the nef cone of $P_1'\times P_2'$ as in Proposition \ref{nef}. Let $N:=\varsigma_1\otimes \varsigma_2$ and $N':=\varsigma_{_3}\otimes \varsigma_{_4}$. Then $\Psi^{*}N$, $\Psi^*N'$ lie in the boundary of the nef cone of $P_1\times P_2$. Note that, for sufficiently large $n,m$, the image of the morphism $|\Psi^{*}N^n|$ is isomorphic to $P_{_1}'$ and the image of the morphism $|\Psi^{*}N'^m|$ is isomorphic to $P_{_2}'$. 
 
 Now we claim that $\Psi^{*}(N)=\vartheta_{_1} ^a\otimes \vartheta_{_2}^b$ or $\vartheta_{_3}^c\otimes \vartheta_{_4}^d$  for some $a,b,c,d>0$. Otherwise, by Lemma \ref{nefc}, the dimension of the image of $|\Psi^{*}(N)^n|$ will be different from the dimension of $P'_{_1}$. Suppose that $\Psi^{*}(N)=\vartheta_{_1} ^a\otimes \vartheta_{_2}^b$ for some $a,b>0$. Then, by our assumption $Y_{_1}\ncong Y_{_2}$, we have $\Psi^*(N')=\vartheta_{_3}^c\otimes \vartheta_{_4}^d$ for some $c,d>0$. Therefore, by Remark \ref{ima}, for sufficiently large $n,m\gg 0$, the images of the morphisms defined by the linear systems $|\Psi^*N^{n}|$ and $|\Psi^*N'^m|$ will be isomorphic to $P_{_1}$ and $P_{_2}$. %This would imply $P_2\simeq P_1'$ then we would have $M_{_{X_{_2}}}(2,L_{-1})\simeq M_{Y_1}(2,L_1)$ (see the proof of \cite[Theorem 2,page 4,(2.3)]{bbd}  for an argument). Hence $X_2\
%simeq Y_1$. We already know $Y_1\simeq X_1$ therefore $X_1\simeq X_2$, a contradiction to our assumption. Therefore, $\Psi^*(N_{_{_1}})=\theta_{_1} ^a\otimes \theta_{_0}^b$. 
Hence, we have  isomorphisms $\Phi_{_1}:P_1\to P_1'$ and $\Phi_{_2}:P_{_2}\to P_{_2}'$ such 
that the following diagrams commute:     
                  \begin{equation}
                     \xymatrix{P_1\times P_2 \ar[r]^{\Psi} \ar[d]^{|\Psi^{*}(N)^n|} & P_1'\times P_2' \ar[d]^{|N^{n}|} \\
                     P_1 \ar[r]^{\Phi_{_1}} & P_1'}
                    \end{equation}
              \begin{equation}\xymatrix{P_1\times P_2 \ar[r]^{\Psi} \ar[d]^{|\Psi^{*}(N')^m|} & P_1'\times P_2' \ar[d]^{|N'^{m}|} \\
                     P_2 \ar[r]^{\Phi_{_2}} & P_2'}
                    \end{equation}

  Therefore, by \cite[Theorem 2.1]{bbd}, there is an isomorphism $f_{_1}:X_{_1}\to Y_1$ such that $f_{_1}(x_{_1})=y_{_1}$ and an isomorphism $f_{_2}:X_2\to Y_2$ such that $f_{_2}(x_{_2})=y_{_2}$.  
  
  Suppose that $\Psi^{*}(N)=\vartheta_{_2} ^c\otimes \vartheta_{_4}^d$, $c,d>0$. Then, by similar arguments as above, we can show that $P_{_2}\simeq P_{_1}'$ and $P_{_1}\simeq P_{_2}'$. Therefore, there is an isomorphism $f'_{_1}:X_{_2}\to Y_{_1}$ such that $f'_{_1}(x_{_2})=y_{_1}$ and an isomorphism $f'_{_2}(x_{_1})=y_{_2}$
  Hence, we conclude $X_0\simeq Y_0$. This completes the proof.
\end{proof}
\section{Appendix}
In this section we will continue with the notations of the preleminary section.
Fix an ample line bundle $\mathcal{O}_{_{X_{_0}}}(1)$ on $X_{_0}$. Let $c_{_i}=deg(\mathcal{O}_{_{X_{_0}}}(1)|_{_{X_{_i}}})$, and $a_{_i}=\frac{c_{_i}}{c_{_1}+c_{_2}}$, $i=1,2$. Let $S(r,\chi)$ be the set of all rank $r$, $a=(a_{_1},a_{_2})$ semistable torsion free sheaves on $X_{_0}$ with Euler charachteristic $\chi$. Note that the Hilbert polynomial $P(E,n)=(c_{_1}+c_{_2})n+\chi(E)$ for all $E\in S(r,\chi)$ (this can be easily computed from equation $(2.1.3)$).
\begin{lemma}(\cite[Septieme Partie]{ns})\label{bounded}
  There exists an integer $m_{_0}$ such that-
 \begin{enumerate}
 \item $H^1(F(m))=0$ for all $F\in S(r,\chi)$ and $m\geq m_{_0}$
 \item $F(m)$ is globally generated by its sections for all $F\in S(r,\chi)$.
 \end{enumerate}
\end{lemma}
%In the next subsection we will see the moduli space of rank $1$,semistable torsion free sheaves with respect to certain choice polarisation is isomorphic to the product of the Jacobians.
\subsection{Moduli space of rank \texorpdfstring{$1$}{} torsion free sheaves over a reducible nodal curve}\label{rank1}
In this subsection we prove that the moduli space of rank $1$,semistable torsion free sheaves with respect to certain choice polarisation is isomorphic to the product of the Jacobians.

\subsubsection{Euler Characteristic bounds for rank $1$ semistable sheaves}
Fix three integers $\chi\neq 0$, $\chi_{_1}$ and $\chi_{_2}\neq 1$ with $\chi > \chi_{_i}$ such that $\chi=\chi_{_1}+\chi_{_2}-1$. Let $\mathcal{O}_{_{X_0}}(1)$ be an ample line bundle such that $deg(\mathcal{O}_{_{X_{_0}}}(1)|_{_{X_1}})=\chi_{_1}-1$ and $deg(\mathcal{O}_{_{X_{_0}}}(1)|_{_{X_2}})=\chi_{_2}$. Since $\chi=\chi_{_1}+\chi_{_2}-1$, the Hilbert polynomial $P(L,n)=(n+1)\chi$ for all $L\in S(1,\chi)$.

Let $b_{_1}=\frac{\chi_{_1}-1}{\chi}$ and $b_{_2}=\frac{\chi_{_2}}{\chi}$. In this subsection whenever we say a semistable rank $1$ torsion free sheaf we assume the semistability with respect to the polarisation $b=(b_{_1},b_{_2})$.
\begin{lemma}\label{line}
 Let $L\in S(1,\chi)$ and $(L_{_1},L_{_2},\lambda)\in \stackrel{\to}{C}$ be the unique triple representing $L$. Then $\chi(L_{_i})$, the Euler characteristic of $L_{_i}$, satisfy the following:
 \[\chi_{_1}\leq \chi(L_{_1})\leq \chi_{_1}+1,~ \chi_{_2}-1\leq \chi(L_{_2})\leq \chi_{_2}.\] Moreover if $L$ is semistable and non locally free then we have $\chi(L_{_1})=\chi_{_1}$ and $\chi(L_{_2})=\chi_{_2}$.
 
 Conversely, suppose $L$ be a rank $1$ torsion free sheaf with $\chi(L_{_i})$ satisfy the above conditions then $L\in S(1,\chi)$.
 \begin{proof}
 By Lemma \cite[Lemma 5.2]{ns} we can easily derive the following: if $L$ is a rank $1$,locally free and $(L_{_1},L_{_2},\lambda) \in \stackrel{\to}{C}$ be the unique triple representing $L$ then we only have to check the semistability condition for the subtriples  $(L_{_1}(-p),0,0)$ and $(0,L_{_2},0)$. If $L$ is a rank $1$, non locally free sheaf and $(L_{_1},L_{_2},0)$ be the triple representing $L$ then we only have to check the semistability for the subtriples $(L_{_1},0,0)$ and $(0,L_{_2},0)$. Now by using the definition of semistability (see \ref{semi}) we immediately get the above Lemma.
\end{proof}
\end{lemma}
Fix an integer $m\geq m_{_0}$ such that Lemma \ref{bounded} holds for all $F\in S(1,\chi)$ and let $P(n)=(n+1)\chi$. Let $Q(1,\chi)$ be the Quot scheme parametrising all coherent quotients \[\mathcal{O}_{_{X_{_0}}}^{\oplus p(m)}\to L\to 0\] with Hilbert polynomial $P(n)$ and $\mathcal{U}^1$ be the universal quotient sheaf of $\mathcal{O}_{_{X_{_0}\times Q(1,\chi)}}^{\oplus p(m)}$ on $X_{_0}\times Q(1,\chi)$. Let $R(1,\chi)^{ss}$ be the open subset of $Q(1,\chi)$ such that if $q\in R(1,\chi)^{ss}$ then $\mathcal{U}^1_{_q}:=\mathcal{U}^1|_{_{X_{_0}\times q}}$ is a rank $1$ semistable torsion free quotient and the natural map \[H^0(\mathcal{O}_{_{X_{_0}\times q}})\to H^0(\mathcal{U}^1_{_q})\] is an isomorphism. Note that if $L\in S(1,\chi)$ then ,by Lemma \ref{bounded}, $L(m)$ is globally generated. Thus $L(m)$ is a quotient of a  trivial sheaf of rank $p(m):=h^0(L(m))$ and the natural map $H^0(\mathcal{O}_{_{X_{_0}}}^{\oplus p(m)})\to H^0(L(m))$ is an isomorphism. Therefore $L(m)\simeq \mathcal{U}^{1}_{_q}$ for some $q\in R(1,\chi)^{ss}$. The group $GL(p(m))$ acts on $Q(1,\chi)$ and $R(1,\chi)^{ss}$ is invariant under the action of $GL(p(m))$. Moreover, the action of $GL(p(m))$ goes down to an action of $PGL(p(m))$. By a general result in (see \cite[Septieme partie,III, Theorem 15]{se1}) the good quotient $R(1,\chi)^{ss} \parallelslant PGL(p(m))$ exists as a reduced, projective scheme. Let $R_{_0}$ be the open subset of $R(1,\chi)^{ss}$ consisting of only rank $1$, locally free sheaves. Then $R_{_0}=R_{_1}\sqcup R_{_1}'$ where $R_{_1}$ consists of those rank $1$ locally free sheaves $L$ such that $\chi(L_{_1})=\chi_{_1}$, $\chi(L_{_2})=\chi_{_2}$ and $R_{_1}'$ consists of those rank $1$ locally free sheaves $L$ such that $\chi(L_{_1})=\chi_{_1}+1$, $\chi(L_{_2})=\chi_{_2}-1$. Let $J^{\chi_{_i}}(X_{_i})$ be the Jacobian of isomorphism classes of line bundles over $X_{_i}$ with Euler characteristic $\chi_{_i}$, $i=1,2$. With these notations the main theorem of this subsection is:
\begin{theorem}\label{moduli}
 The good quotient $R(1,\chi)^{ss}\parallelslant PGL(p(m))$ is isomorphic to $J^{\chi_{_1}}(X_{_1})\times J^{\chi_{_2}}(X_{_2})$.
 \begin{proof}
 Let $q:R(1,\chi)^{ss}\to R(1,\chi)^{ss}\parallelslant PGL(p(m))$ be the quotient map. Note that, since $Hom(L,L)=\C$ for all $L\in R_{_1}$, $PGL(p(m))$ acts freely on $R_{_1}$. Moreover, $R_{_1}$ is smooth and irreducible. Therefore, the quotient $R_{_1}/PGL(p(m))$ is smooth and irreducible (see \cite[Corollary 4.2.13]{huy}). 
 \noindent  
{\em Claim 1:} The quotient $q(R_{_1})=R_{_1}/PGL(p(m))$ is isomorphic to $J^{\chi_{_1}}(X_{_1})\times J^{\chi_{_2}}(X_{_2})$. 

To see this consider $\mathcal{U}$ over $X_{_0}\times R_{_1}$. Then $\mathcal{U}$ is locally free and hence $\mathcal{U}_{_i}=\mathcal{U}|_{_{X_{_i}}\times R_{_1}}$ is locally free. Moreover, $\chi(\mathcal{U}_{_i}|_{_{X_{_i}\times q}})=\chi_{_i}$, $i=1,2$. Thus by the universal property of $J^{\chi_{_i}}(X_{_i})$ we get a morphism $f_{_i}:R_{_1}\to J^{\chi_{_i}}$, $i=1,2$. Therefore, we get a morphism $f=(f_{_1},f_{_2}):R_{_1}\to J_{_0}=J^{\chi_{_1}}(X_{_1})\times J^{\chi_{_2}}(X_{_2})$. Clearly, this morphism is $PGL(p(m))$-invariant and the fibres of this morphism are isomorphic to the orbits of the $PGL(p(m))$ action. Therefore, we get a bijective morphism $R_{_1}/PGL(p(m))\to J_{_0}$. Since $R_{_1}/PGL(p(m))$, $J_{_0}$ are integral and $J_{_0}$ is smooth, we have $R_{_1}/PGL(p(m))$ is isomorphic to the variety $J_{_0}=J^{\chi_{_1}}(X_{_1})\times J^{\chi_{_2}}(X_{_2})$. 
\noindent
{\em Claim 2:} We have an equality:
\[R_{_1} / PGL(p(m))=R(1,\chi)^{ss} \parallelslant PGL(p(m).\]
Note that Claims 1 and 2 together prove the theorem.

Let $[L_{_0}]\in R(1,\chi)^{ss}\parallelslant PGL(p(m)$ where $[]$ is orbit closure equivalence class. Then we want to show there is a $L\in R_{_1}$ such that $q(L)=[L_{_0}]$. In other words the orbit closure $\overline{O(L)}$ intersects the orbit closure $\overline{O(L_{_0})}$. Suppose, $L_{_0}$ is locally free with $\chi(L_{_0}|_{_{X_{_i}}})=\chi_{_i}$, $i=1,2$ then there is nothing to prove. So we assume that $L_{_0}$ is a rank $1$ torsion free but non-locally free sheaf. Then as $L_{_0}$ is semistable, by Lemma \ref{line}, we get $\chi(L_{_i})=\chi_{_i}$, $i=1,2$ where $(L_{_1},L_{_2},0)\in \stackrel{\to}{C}$ is the unique triple representing $L_{_0}$. Let $L$ be the  rank $1$ locally free sheaf corresponding to the triple $(L_{_1},L_{_2},\lambda)\in \stackrel{\to}{C}$ where $\lambda:L_{_1}(p)\to L_{_2}(p)$ is an isomorphism. We will show now the orbit closure $\overline{O(L)}$ intersects the orbit $O(L_{_0})$. For this, let $p_{_i}:X_{_i}\times \mathbb{A}^1\to X_{_i}$, $i=1,2$, be the two projections. We again denote the pullback $p_{_i}^*L_{_i}$ by $L_{_i}$. Since $L_{_i}$, $i=1,2$, are  free $\mathcal{O}_{_{\mathbb{A}^1}}$-module, we can choose a $\mathcal{O}_{_{\mathbb{A}^1}}$-module homomorphism $\lambda:L_{_1}|_{_{p\times \mathbb{A}^1}}\to L_{_2}|_{_{p\times \mathbb{A}^1}}$ such that $\lambda(t):L_{_1}(p,t)\to L_{_2}(p,t)$ is an isomorphism for all $t\neq 0$ and $\lambda(0)=0$. Let $G$ be the graph of the morphism $\lambda$ in $L_{_1}|_{_{p\times \mathbb{A}^1}}\oplus L_{_1}|_{_{p\times \mathbb{A}^1}}$ and $G':=\frac{L_{_1}|_{_{p\times \mathbb{A}^1}}\oplus L_{_2}|_{_{p\times \mathbb{A}^1}}}{G}$. Let $\mathcal{L}:=Ker(L{_1}\oplus L_{_2}\to G')$ over $X_{_0}\times \mathbb{A}^1$. Now $L_{_1}\oplus L_{_2}$ and $G'$, being free $\mathcal{O}_{_{\mathbb{A}}}^1$-module, are flat over $\mathbb{A}^1$. Therefore, $\mathcal{L}$ is flat over $\mathbb{A}^1$. We also see that $\mathcal{L}_{_t}$ is the torsion free sheaf corresponding to the triple $(L_{_1},L_{_2},\lambda(t))$. Therefore, $\mathcal{L}_{_t}\simeq L$ for all $t\neq 0$ and $\mathcal{L}_{_0}\simeq L_{_0}$. Note that, as $\mathcal{L}_{_t}\in R(1,\chi)^{ss}$ for all $t\in \mathbb{A}^1$, $H^1(\mathcal{L}_{_t})=0$ and $\mathcal{L}_{_t}$ is globally generated for all $t\in \mathbb{A}^1$. By semicontinuity theorem, we get ${p_{_2}}_*\mathcal{L}$ is locally free sheaf of rank $p(m)$ on $\mathbb{A}^1$. Since any locally free sheaf on $\mathbb{A}^1$ is free, ${p_{_2}}_*\mathcal{L}\simeq {\mathcal{O}_{_{\mathbb{A}^1}}}^{\oplus p(m)}$. Thus we get a quotient \[\mathcal{O}_{_{X_{_0}\times \mathbb{A}^1}}^{\oplus p(m)}\simeq p_{_2}^*{p_{_2}}_*\mathcal{L} \to \mathcal{L}\to 0.\] such that $H^0(\mathcal{O}_{_{X_0\times t}}^{p(m)})\to H^0(\mathcal{L}_{
_t})$ is an isomorphism for all $t\in \mathbb{A}^1$. Hence we get a morphism $\phi:\mathbb{A}^1\to R(1,\chi)^{ss}$ such that $\phi^*\mathcal{U}^1\simeq \mathcal{L}$. Since $\mathcal{L}_{_t}\simeq L$ for all $t\in \mathbb{A}^1-0$, $\phi(\mathbb{A}^1-0)$ lies in the $PGL(p(m))$ orbit of $L$ and $\phi(0)=L_{_0}$. Therefore, $L_{_0}$ is in the orbit closure of $L$. Clearly, $\chi(L_{_i})=\chi(L_{_i}|_{_{X_{_i}}})=\chi_{_i}$, $i=1,2$. Thus $L\in R_{_1}$ and we are done. 
 
 Finally suppose, $L_{_0}$ is rank $1$, locally free sheaf such that $\chi(L_{_1})=\chi_{_1}+1$ and $\chi(L_{_2})=\chi_{_1}-1$ where $(L_{_1},L_{_2},\lambda)$, $\lambda:L_{_2}(p)\to L_{_1}(p)$ an isomorphism, is the unique triple representing $L_{_0}$. Let $L$ is the rank $1$ locally free sheaf represented by $(L_{_1}(-p),L_{_2}(p),\lambda)\in \stackrel{\to}{C}$. Then the orbit closure $\overline{O(L_{_0})}$ intersects the orbit closure $\overline{O(L)}$ in $R(1,\chi)^{ss}$. This easily follows from the observation: The torsion free sheaf $L'$ represented by the triple $(L_{_1},L_{_2},0)\in \stackrel{\leftarrow}{C}$ is in the orbit closure of $L_{_0}$. Note, by Remark \ref{direction}, $L'$ is isomorphic to the torsion free sheaf represented by $(L_{_1}(-p),L_{_2}(p),0)\in \stackrel{\to}{C}$. Thus , by Lemma \ref{line}, $L'$ is semistable and is also in the orbit closure of $L$. Thus given any $[L_{_0}]\in R(1,\chi)^{ss}\parallelslant PGL(p(m)$ we have seen that there is a $L\in R_{_1}$ such that $q(L)=[L_{_0}]$.
  
   %Since each orbit closu
\end{proof}
\end{theorem}

\subsection{Determinant morphism}
Fix an odd integer $\chi$ and a polarisation $(a_{_1},a_{_2})$ on $X_{_0}$ such that $a_{_1}\chi$ is not an integer. We also fix an integer $m'$ such that Lemma \ref{bounded} holds for all $E\in S(2,\chi)$ . Let $Q(2,\chi)$ be the Quot scheme parametrising all coherent quotients \[\mathcal{O}_{_{X_{_0}}}^{\oplus p(m')}\to E\to 0\] and $\mathcal{U}^2$ be the universal quotients sheaf of $\mathcal{O}_{_{X_{_0}\times Q(2,\chi)}}^{\oplus p(m')}$ on $X_{_0}\times Q(2,\chi)$. Let $R(2,\chi)^{ss}$ be the open subset of $Q(2,\chi)$ such that if $q\in R(2,\chi)^{ss}$ then $\mathcal{U}^2_{_q}:=\mathcal{U}^2|_{_{X_{_0}\times q}}$ is a rank $2$ semistable torsion free quotient and the natural map \[H^0(\mathcal{O}_{_{X_{_0}\times q}})\to H^0(\mathcal{U}^1_{_q})\] is an isomorphism. The moduli space $M(2,a,\chi)$ is isomorphic to the quotient $R(2,\chi)^{ss}/PGL(p(m'))$. Let $M_{_{12}}$ and $M_{_{21}}$ be the two smooth components of $M(2,a,\chi)$. Let $M_{_{12}}^{0}\subset M_{_{12}}$ (resp. $M_{_{21}}^{0}\subset M_{_{21}}$) be the open subvariety of $M_{_{12}}$ (resp. $M_{_{21}}$) consisting of isomorphism classes of rank $2$ semistable locally free sheaves.

%Consider the open subvariety $R(\chi')^{ss}$ of the Quote scheme $Q(\chi')$ (see section \ref{rank1}) consisting of rank $1$, torsion free sheaves which are semistable with respect to the polarisation $b=(b_{_1},b_{_2})$ where $b_{_1}=\frac{\chi_{_1}'}{\chi'}$, $b_{_2}=\frac{\chi_{_2}'-1}{\chi'}$ and $\chi':=\chi-(1-g)$.
%Then by Theorem \ref{moduli},  $R(\chi')^{ss}/PGL(E)$ is isomorphic to $J^{\chi_{_1}'}(X_{_1})\times J^{\chi_{_2}'}(X_{_2})$ where $\chi_{_i}':=\chi_{_i}-(1-g_{_i})$, $i=1,2$.
\begin{proposition}\label{determin}
There exists a determinant morphism $det:M(2,a,\chi)\to J^{\chi_{_1}'}(X_{_1})\times J^{\chi_{_2}'}(X_{_2})$ where $\chi_{_i}'=\chi_{_i}-(1-g_{_i})$, $i=1,2$.
 \begin{proof}
Let $R_{_2}$ be the open subset of $R(2,\chi)^{ss}$ such that for all $q\in R_{_2}$, $\mathcal{U}_{_q}$ is rank $2$ semistable locally free and $\chi(\mathcal{U}_{_q}|_{_{X_{_i}}})=\chi_{_i}$, $i=1,2$. Then $M_{_{12}}^{0}\simeq R_{_2}/PGL(p(m')$. Let $R_{_1}$ be the open subset of $R(1,\chi')^{ss}$, where $\chi':=\chi-(1-g)$, such that for all $q\in R_{_1}$, $\mathcal{U}^1_{_q}$ is rank $1$, semistable locally free and $\chi(\mathcal{U}^1_{_q}|_{_{X_{_i}}})=\chi_{_i}-(1-g_{_i})$, $i=1,2$. Let us restrict the universal quotient sheaf $\mathcal{U}^2$ on $X_{_0}\times R_{_2}$. Then $\mathcal{U}^2$ is a rank $2$ locally free shef on $X_{_0}\times R_{_2}$ such that $\mathcal{U}^2_{_q}$ is semistable and $\chi(\mathcal{U}^2_{_q}|_{_{X_i}})=\chi_{_i}$ for all $q\in R_{_2}$, $i=1,2$. Thus $\wedge^2\mathcal{U}^2$ is a flat family of rank $1$ locally free sheaves on $X_{_0}\times R_{_2}$ such that $\chi(\wedge^2\mathcal{U}^2_{_q}|_{_{X_i}})=\chi_{_i}-(1-g_{_i})$ for all $q\in R_{_2}$, $i=1,2$. By Lemma \ref{line} $\wedge^2\mathcal{U}^2_{_q}$ is semistable for all $q\in R_{_2}$. By Lemma \ref{bounded} there exists an integer $m$ such that $H^1(\wedge^2\mathcal{U}^2_{_q}(m))=0$ and $\wedge^2\mathcal{U}^2_{_q}(m)$ is globally generated for all $q\in R_{_2}$. Therefore, there is an open covering $\{U_{_i}\}$ of $R_{_2}$ and morphisms $det_{_i}:U_{_i}\to R_{_1}$ such that, for any non-empty open set $U_{_{ij}}:=U_{_i}\cap U_{_j}$ if we denote by $det_{{ij}}=det_{_i}|_{_{U_{_{ij}}}}$, then there exists $g\in PGL(n)(U_{_{ij}})$ with the property $det_{_{ij}}=gdet_{_{ji}}$, where $n=h^0(\wedge^2\mathcal{U}^2_{_q}(m))$. Therefore, we get a well-defined morphism $det:R_{_2}\to R_{_1}/PGL(n)=J^{\chi_{_1}'}(X_{_1})\times J^{\chi_{_2}'}(X_{_2})$.
Since $M_{_{12}}$ is a smooth projective variety and $J_{_0}$ is an abelian variety the  morphism $det_{_1}^0$ extends to a morphism $det_{_1}:M_{_{12}}\to J_{_0}:=J^{\chi_{_1}'}(X_{_1})\times J^{\chi_{_2}'}(X_{_2})$. By similar arguments we get a morphism $det_{_2}:M_{_{21}}\to J_{_0}':=J^{\chi_{_1}'+1}(X_{_1})\times J^{\chi_{_2}'-1}(X_{_2})$. Let $F\in M_{_{12}}\cap M_{_{21}}$. Then $F$ is represented by a unique triple $(F_{_1},F_{_2},A)\in \stackrel{\to}{C}$ where $rk(A)=1$. $F$ is also represented by another triple $(F_{_1}',F_{_2}',B)\in \stackrel{\leftarrow}{C}$ such that $F_{_i}$ and $F_{_i}'$ are related by the diagram in Remark \ref{direction}. Note that $\wedge^2F_{_1}'\simeq \wedge^2F_{_1}(p)$ and $\wedge^2F_{_2}'\simeq \wedge^2F_{_2}(-p)$. Claim $det_{_1}(F)=(\wedge^2F_{_1},\wedge^2F_{_2})$ and $det_{_2}(F)=(\wedge^2F_{_1}',\wedge^2F_{_2}')$. First we construct a flat family $\mathcal{F}$ over $X\times \mathbb{A}^1$ such that $\mathcal{F}_{_t}$ is rank $2$, semistable locally free sheaf for all $t\neq 0$ and $\mathcal{F}_{_0}\simeq F$. Moreover, $\mathcal{F}_{_t}|_{_{X_{_1}}}\simeq F_{_1}$ and $\mathcal{F}_{_t}|_{_{X_{_2}}}\simeq F_{_2}$ for all $t\neq 0$ (this can be done using the similar construction given in the proof of Theorem \ref{moduli}). Thus we get a morphism $\phi:\mathbb{A}^1\to M_{_{12}}$ such that $\phi(\mathbb{A}^1-0)\subset M_{_{12}}^0$ and $\phi(0)=F$. Denote the restriction of $\mathcal{F}$ over $X_{_0}\times(\mathbb{A}^1-0)$ by $\mathcal{F}'$. Then $\wedge^2\mathcal{F}'$ induces  a morphism $\wedge^2\phi:\mathbb{A}^1-0\to J_{_0}$. Since $\mathcal{F}_{_t}|_{_{X_{_i}}}\simeq F_{_i}$, $i=1,2$, we get $\wedge^2\mathcal{F}'_{_t}\simeq L$ for all $t\neq 0$ where $L$ is the line bundle represented by the triple $(\wedge^2F_{_1},\wedge^2F_{_2},\lambda)$. Thus $\wedge^2\phi$ is constant for all $t\neq 0$. Thus $\wedge^2\phi$  extends as a constant morphism over whole $\mathbb{A}^1$ and $\wedge^2\phi(t)=(\wedge^2F_{_1},\wedge^2F_{_2})$ for all $t\in \mathbb{A}^1$ . Clearly, $det_{_1}(\phi(t))=\wedge^2\phi(t)$ for all $t\in \mathbb{A}^1$.  By similar arguments we can show that $det_{_2}(F)=(\wedge^2F_{_1}',\wedge^2F_{_2}')$. Now we define a morphism $det:M(2,a,\chi)\to J_{_0}$ in the following way:
 \[det(F):=det_{_1}(F) ~ \mbox{if} ~ F\in M_{_{12}}\]
\[det(F):=f^{-1}(det_{_2}(F)) ~ \mbox{if} ~ F\in M_{_{21}}\]
where $f:J_{_0}\to J_{_0}'$ is an isomorphism defined by the association $(L_{_1},L_{_2})\to (L_{_1}(p),L_{_2}(-p))$. By the above discussion clearly, $det$ is a well defined morphism.
\end{proof}
\end{proposition}
\begin{proposition}\label{tran}
The fibres of the morphism $det:M(2,a,\chi)\to J^{\chi_{_1}'}(X_{_1})\times J^{\chi_{_2}'}(X_{_2})$ are the union of two smooth, irreducible projective varieties meeting transversally along a smooth divisor.
\begin{proof}
Let $J^0(X_{_0})$ be the variety parametrising all isomorphism classes of line bundles $L$ such that $deg(L|_{X_{_i}})=0$, $i=1,2$. Then we can show that $J^0(X_{_0}):=J^{0}(X_{_1})\times J^{0}(X_{_2})$ where $J^{0}(X_{_i})$ are the Jacobians of $X_{_i}$, $i=1,2$. Now $J^0(X_{_i})$ acts on $M(2,a,\chi)$ by $F\to F\otimes L$. Clearly, both the components of $M(2,a,\chi)$ and the divisor $D$ are fixed by this action. Also we can easily check that the morphism $det$ is compatible with action of $J^0(X_{_0})$ where action of $J^0(X_{_0})$ on $J_{_0}$ is given by $(\eta_{_1},\eta_{_2})\mapsto (\eta_{_1}\otimes L_{_1},\eta_{_2}\otimes L_{_2})$. Now $det|_{_{M_{_{12}}}}=det_{_1}$ and $det|_{_{M_{_{21}}}}=f^{-1}odet_{_2}$. Clearly, the morphism $det_{_1}:M_{_{12}}\to J_{_0}$ is compatible with the action of $J^0(X_{_0})$. Thus it is a smooth morphism. Since both $M_{_{12}}$ and $J_{_0}$ are smooth we conclude that the fibres of $det_{_1}$ are smooth. Also, by the same reasoning, the fibres of $det_{_1}|_{_D}$ , the restriction to the divisor $D$, are also smooth. Thus the fibres of $det_{_1}$ intersects $D$ transversally. By the same arguments we can show that the fibres of $det_{_2}$ intersects $D$ transversally. Hence,  we conclude that the intersection of $det_{_1}^{-1}(\xi)$ with $det_{_2}^{-1}(f(\xi))$, $\xi\in J_{0}$ is smooth. Therefore, the fibres of $det$ are the union of two smooth, projective varieties intersecting transversally.
 
 %Let $M_{_1}$ and $M_{_2}$ be the components of $M(2,a,\chi)$ and $det_{_i}$ be the restrictions $det|_{_{M_{_i}}}$, $i=1,2$. Then all the fibres of $det_{_i}$ are isomorphic. We only show this for $det_{_1}$. For $det_{_2}$ the proof is same. Note that if $L$ be an invertible sheaf over $X_{_0}$ such that $deg(L|_{_{X_{_i}}}=0$, $i=1,2$. Then $F\to F\otimes L$ gives an automorphism of $M(2,a,\chi)$. Thus $J^0(X_{_1})\times J^{0}(X_{_2})$ acts on $M(2,a,\chi)$. Let $\xi,\xi'\in J^{\chi_{_1}'}(X_{_1})\times J^{\chi_{_2}'}(X_{_2})$ and $M_{_{1,\xi}}:=det^{-1}(\xi)$, $M_{_{1,\xi'}}:=det^{-1}(\xi')$.  Tensoring by a suitable deg $0$ invertible sheaf $L$ will induce an isomorphism between $M_{_{1,\xi}}^0$ and $M_{_{1,\xi'}}^0$. Clearly, this isomorphism extends to whole of $M_{_{1,\xi}}$. Thus, any two fibres of $det_{_1}$ are isomorphic. Similarly, any two fibres of $det_{_2}$ are isomorphic. As both $M_{_1}$ and $J^{\chi_{_1}'}(X_{_1})\times J^{\chi_{_2}'}(X_{_2})$ are smooth the generic fibres of $det_{_1}$ are smooth. Hence, all the fibres of $det_{_1}$ are smooth. Similarly all the fibres of $det_{_2}$ are smooth.
\end{proof}
\end{proposition}
\subsection{Relative moduli space and relative determinant morphism}
 Let $C=SpecR$ where $R$ is a complete discrete valuation ring and $\mathcal{X}\to B$ be a flat family of proper, connected curves. We assume the generic fibre $\mathcal{X}_{_{\eta}}$ is smooth and the closed fibre $\mathcal{X}_{_0}$ is the curve $X_{_0}$. We further assume that $\mathcal{X}$ is regular over $\C$. For any $C$ scheme $S$ we denote $\mathcal{X}\times_{_C} S$ by $\mathcal{X}_{_S}$. Fix an integer $\chi$.
 
 \subsubsection{Relative moduli of rank $1$, torsion free sheaves}\label{degen}:
 Fix a relatively ample line bundle $\mathcal{O}_{_{\mathcal{X}}}(1)$ over $\mathcal{X}$ such that $\mathcal{O}_{_{\mathcal{X}}}(1)|_{_{X_0}}$ gives the polarisation of type $(b_{_1},b_{_2})$. Let $\mathcal{Q}_{_1}\to C$ be the relative Quot scheme parametrising all rank $1$ coherent quotients \[\mathcal{O}_{_{\mathcal{X}}}^{p(N)}\to \mathcal{L}\to 0.\] which has the fixed Hilbert polynomial $p(n):=(n+1)\chi'$, $\chi'=\chi-(1-g)$, along the fibre of $\mathcal{X}$ and flat over $C$. Let $\mathcal{U}$ be the universal quotient sheaf of $\mathcal{O}_{_{\mathcal{X}_{_{\mathcal{Q}_1}}}}^{\oplus p(N)}$ on $\mathcal{X}_{_{\mathcal{Q}_1}}$.  Let $\mathcal{G}=Aut(\mathcal{O}_{_{\mathcal{X}}}^{p(N)})$ be the reductive group scheme over $C$. Then $\mathcal{G}$ acts on $\mathcal{Q}_{_1}$. Let $\mathcal{R}_{_1}^{ss}$ be the open subvariety of $\mathcal{Q}_{_1}$ consisting of those quotients $\mathcal{L}$ which are semistable along the fibre of $\mathcal{X}$ and the natural map $H^0(\mathcal{O}_{_{\mathcal{X}}}^{p(N)})\to H^0(\mathcal{L})$ is an isomorphism. We can construct a good quotient $\mathcal{J}:=\mathcal{R}_{_1}^{ss}\parallelslant\mathcal{G}$, projective over $C$ using GIT over arbitrary base. Also note that $({\mathcal{R}_{_1}}^{ss}\parallelslant\mathcal{G})_{_t}={\mathcal{R}_{_1}}^{ss}_{_t}\parallelslant\mathcal{G}_{_t}$ for all $t\in C$ (\cite[Theorem 4]{se2}). Thus the general fibre $\mathcal{J}_{_{\eta}}$ is the Jacobian $J^{\chi'}(\mathcal{X}_{_{\eta}})$ and by Theorem \ref{moduli} the closed fibre $\mathcal{J}_{_0}$ is isomorphic to $J^{\chi'_{_1}}(X_{_1})\times J^{\chi'_{_2}}(X_{_2})$. %Let $\mathcal{R}_{_0}=\mathcal{R}^{ss}\setminus \overline{R_{_2}}$ where $\overline{R_{_2}}$ is the closure  of $R_{_2}$ in $R^{ss}$ (see subsection \ref{rank1} for the notation). Then $\mathcal{R}_{_0}$ is smooth over $\C$. Since $P\mathcal{G}$ acts freely on $\mathcal{R}_{_0}$ we have $\mathcal{R}_{_0}/P\mathcal{G}$ a locally trivial quotient and hence is regular over $\C$.  By same argument as in the proof of \ref{moduli} we can show that $\mathcal{J}\simeq \mathcal{R}_{_0}/P\mathcal{G}$. Thus $\mathcal{J}$ is regular over $\C$. Hence $\mathcal{J}$ is flat over $C$ as a torsion free module over a d.v.r is flat. By the above discusion we conclude that there exists a proper flat family  $\mathcal{J}\to C$ such that $\mathcal{J}_{_{\eta}}=J^{\chi'}(X_{_{\eta}})$ and $\mathcal{J}_{_0}=J^{\chi'_{_1}}(X_{_1})\times J^{\chi'_{_2}}(X_{_2})$. Moreover, $\mathcal{J}$ is regular over $\C$.

 \subsubsection{Relative moduli of rank $2$, torsion free sheaves}:
 Fix a relatively ample line bundle $\mathcal{O}_{_{\mathcal{X}}}(1)'$ over $\mathcal{X}$ such that $\mathcal{O}_{_{\mathcal{X}}}(1)'|_{_{X_0}}=\mathcal{O}_{_{X_{_0}}}(1)$ gives the polarisation of type $(a_{_1},a_{_2})$ such that $a_{_1}\chi$ is not an integer. Let $\mathcal{Q}_{_2}\to C$ be the relative Quot scheme parametrising all rank $2$ coherent quotients \[\mathcal{O}_{_{\mathcal{X}}}^{p(N)}\to \mathcal{E}\to 0.\] which has the fixed Hilbert polynomial $p(n):=(c_{_1}+c_{_2})n+\chi$, where $c_{_i}=deg(\mathcal{O}_{_{X_{_0}}}(1)|_{_{X_{_i}}}$, along the fibre of $\mathcal{X}$ and flat over $C$.Let $\mathcal{U}''$ be the universal quotient sheaf of $\mathcal{O}_{_{\mathcal{X}_{_{\mathcal{Q}_1}}}}^{\oplus p(m')}$ on $\mathcal{X}_{_{\mathcal{Q}_2}}$. Let $\mathcal{G}'=Aut(\mathcal{O}_{_{\mathcal{X}}}^{p(N)})$ be the reductive group scheme over $C$. Then $\mathcal{G}'$ acts on $\mathcal{Q}_{_2}$. Let $\mathcal{R}_{_2}^{ss}$ be the open subvariety of $\mathcal{Q}_{_2}$ consisting of those quotients $\mathcal{E}$ which are semistable along the fibre of $\mathcal{X}$ and the natural map $H^0(\mathcal{O}_{_{\mathcal{X}}}^{p(N)})\to H^0(\mathcal{E})$ is an isomorphism. It is shown in \cite[Theorem 4.2]{ns} a relative moduli space $\mathcal{M}:=\mathcal{R}_{_2}^{ss}\parallelslant\mathcal{G}'$ exists and it is projective over $C$ using GIT over arbitrary base.  Thus the general fibre $\mathcal{M}_{_{\eta}}$ is the moduli space $M_{_{\mathcal{X}_{_\eta}}}(2,\chi)$ of rank $2$, semistable sheaves with Euler characteristic $\chi$ and $\mathcal{M}_{_0}$ is the moduli space $M(2,a,\chi)$.  Note that, if $\mathcal{X}$ is a regular surface, by \cite[Remark 4.2]{ns}, $\mathcal{R}_{_2}^{ss}$ is smooth over $\C$. If we assume $\chi$ to be odd then $\mathcal{R}_{_2}^{ss}=\mathcal{R}_{_2}^{s}$. Therefore, $P\mathcal{G}'$ acts on $\mathcal{R}_{_2}^{s}$ freely. Since $\mathcal{R}_{_2}^{ss}$ is smooth we conclude that $\mathcal{M}=\mathcal{R}_{_2}^{s}/P\mathcal{G}'$ is regular over $\C$ (see \cite[Corollary 4.2.23]{huy}). %By the above discussion we
 
\begin{proposition}\label{reldet}
  There exists a morphism $Det:\mathcal{M}\to \mathcal{J}$ such that the following diagram commutes-
  \begin{equation}\label{det2}
  \xymatrix{
\mathcal{M} \ar[rr]^{Det} \ar[rd]_{\pi'} && \mathcal{J}\ar[ld]^{\pi''} \\
&C
}
\end{equation}
Moreover, we have $Det|_{_{\mathcal{M}_{_0}}}=det$.
\begin{proof}
Let $\mathcal{R}_{_2}^0$ be the open subscheme of $\mathcal{R}_{_2}^{ss}$ such that if $q\in \mathcal{R}_{_2}^0$ then $\mathcal{U}''_{_q}$ is rank $2$,locally free. Then by similar arguments as in the proof of Proposition \ref{determin} we get a morphism $Det'^{0}:\mathcal{R}_{_2}^0\to \mathcal{J}$ and this descends to a morphism $Det^0:\mathcal{M}^0=\mathcal{R}_{_2}^0/P\mathcal{G}'\to \mathcal{J}$. Let $Z=\mathcal{M}\setminus \mathcal{M}^0$. Then $Z$ is supported on the fibre $\mathcal{M}_{_0}$ and $\mathcal{M}_{_0}\setminus Z$ is a dense open set. Clearly, $Det^0|_{_{\mathcal{M}_{_0}\setminus Z}}=det$. Thus if we can show that $Det^0$ extends as a morphism $Det:\mathcal{M}\to \mathcal{J}$ then we have $Det|_{_{\mathcal{M}_{_0}}}=det$.
 Let $\Gamma$ be the graph of the morphism $Det^0$ in $\mathcal{M}\times_{_C} \mathcal{J}$ and $\overline {\Gamma}$ be the Zariski closure of $\Gamma$ in $\mathcal{M}\times_{_C} \mathcal{J}$. Let $p_{_1}$, $p_{_2}$ be the restriction of the two projections to $\overline{\Gamma}$. Then the morphism $p_{_1}:\overline {\Gamma}\to \mathcal{M}$ is clearly birational. We will show that it is bijective. Since $\mathcal{M}$ is smooth over $\C$ it will follow that $p_{_1}$ is an isomorphism. Thus we define $Det:=p_{_2}p_{_1}^{-1}$ which clearly extends the morphism $Det^0$. Let $(F,G_1)$ and $(F,G_2)$ be the closed points of $\overline {\Gamma}\setminus \Gamma$. Then we claim that $G_{_1}\simeq G_{_2}$. The claim follows from the following observation:
 
 Let $\tilde R$ be a complete discrete valuation ring such that $\tilde C\to C$ is dominant, where $\tilde C:= Spec \tilde R$. Let $t$ be the generic point of $\tilde C$ and $0$ be the closed point of $\tilde C$,. Let $\mathcal{F}$ be a  coherent sheaf on $\mathcal{X}\times_{_C} \tilde C$, flat over $\tilde C$ such that $\mathcal{F}_{_t}$ is a rank $2$ semistable bundle over $\mathcal{X}_{_t}$ and $\mathcal{F}_{_0}\simeq F$ and $\mathcal{G}$  be a locally free sheaf on $\mathcal{X}_{_{\tilde C}}:=\mathcal{X}\times_{_{ C}} \tilde C$ such that $\mathcal{G}_{_t}\simeq \wedge^2\mathcal{F}_{_t}$ and $\mathcal{G}_{_0}\simeq G_{_1}$. Let $\pi: \tilde {\mathcal{X}}_{_{\tilde C}}\to \mathcal{X}_{_{\tilde C}}$ be the disingularization of $\mathcal{X}_{_{\tilde C}}$ at $p$. The exceptional divisor $\pi^{-1}(p)=\Sigma_{_i}D_{_i}$ is a chain of $(-2)$ curves, and the special fibre $\tilde {\mathcal{X}}_{_0}$ of $\tilde {\mathcal{X}}_{_{\tilde C}}\to \tilde C$ at $0$ is $X_{_1}+X_{_2}+\Sigma D_{_i}$ and $X_{_1}\cap (X_{_2}+\Sigma D_{_i})=p_{_1}$, $X_{_2}\cap(X_{_1}+\Sigma D_{_i})=p_{_2}$. Let $\mathcal{F}''$ be the restriction of $\mathcal{F}$ on $(\mathcal{X}-p)\times_{_C} \tilde C$. Identifying $\tilde {\mathcal{X}}_{_{\tilde C}}\setminus \pi^{-1}(p)\simeq \mathcal{X}_{_{\tilde C}}\setminus {p}=(\mathcal{X}-p)_{_{\tilde C}}$, we can extend the line bundle $\wedge^2\mathcal{F}''$ into a line bundle $\overline {\wedge^2\mathcal{F}''}$ on $\tilde {\mathcal{X}}_{_{\tilde C}}$. This is possible since $\tilde{\mathcal{X}}_{_{\tilde C}}$ is non singular. Therefore, we can show $Pic(\tilde{\mathcal{X}}_{_{\tilde C}})=Pic(\mathcal{X}_{_{ \tilde C}}\setminus p)\oplus_{_i} \Z D_{_i}$. Clearly, we have $\overline {\wedge^2\mathcal{F}''}_{_t}\simeq \pi^*\mathcal{G}_{_t}$. Thus one has \[\overline {\wedge^2\mathcal{F}''}\simeq \pi^*\mathcal{G}\otimes \mathcal{O}_{_{\tilde {\mathcal{X}}_{_{\tilde C}}}}(V),\] where $V\subsetneqq \tilde{\mathcal{X}}_{_0}$ is a vertical divisor i.e of the form $\Sigma n_{_i}D_{_i}$. %Since $\mathcal{O}_{_{\tilde {\mathcal{X}}_{_{\tilde C}}}}(V)|_{_{X_{_i}\setminus p_{_i}}}$ is trivial we get $\overline {\wedge^2\mathcal{F}''}|_{_{X_{_i}\setminus p_{_i}}}\simeq \mathcal{G}_{_0}|_{_{X_{_i}\setminus p_{_i}}}$, $i=1,2$. 
 Therefore, $\overline {\wedge^2\mathcal{F}''}|_{_{X_{_i}}}\simeq \pi^*\mathcal{G}_{_0}|_{_{X_{_i}}}(n_{_i}p_{_i})$ for some integers $n_{_i}$, $i=1,2$ since $\mathcal{O}_{_{\tilde {\mathcal{X}}_{_{\tilde C}}}}(V)|_{_{X_{_i}}}=\mathcal{O}_{_{X_{_i}}}(n_{_i}p_{_i})$. Let $e_{_i}=deg(\overline {\wedge^2\mathcal{F}''}|_{_{X_{_i}}})$, $i=1,2$. Then $n_{_i}=e_{_i}-deg(\mathcal{G}_{_0}|_{_{X_{_i}}})$. Therefore the integers $n_{_i}$ only depend on $deg(\mathcal{G}_{_0}|_{_{X_{_i}}})$. Let $L_{_i}=\overline {\wedge^2\mathcal{F}''}|_{_{X_{_i}}}(-n_{_i}p_{_i})$. Then $G_{_1}=\mathcal{G}_{_0}$ is isomorphic to the line bundle $L$ which is uniquely represented by the triple $(L_{_1},L_{_2},\lambda)$. Suppose $\mathcal{G}'$ be another locally free sheaf on $\mathcal{X}_{_{\tilde C}}$ such that $\mathcal{G}'_{_t}=\wedge^2\mathcal{F}_{_t}$ and $\mathcal{G}'_{_0}\simeq G_{_2}$. Since $deg(G_{_1}|_{_{X_{_i}}})=deg(G_{_2}|_{_{X_{_i}}})$ , by the above argument, we can show that $G_{_2}\simeq L$. Thus $G_{_1}\simeq G_{_2}$.
 
% Observe that $\mathcal{F}$ is locally free over $(\mathcal{X}-p)\times_{_C} \tilde C$ where $p$ is the singular point of $\mathcal{X}_{_0}$. Let $\mathcal{F}''$ be the restriction of $\mathcal{F}$ on $(\mathcal{X}-p)\times_{_C} \tilde C$. %Since  $(\mathcal{X})\times_{_C} \tilde C$ is regular over $\C$, the line bundle $\wedge^2\mathcal{F}''$ extends over $\mathcal{X}\times_{_C} \tilde C$. Let us denote this line bundle by $\overline {\wedge^2\mathcal{F}''}$. Note that $\chi(\overline {\wedge^2\mathcal{F}''}_{_x})=\chi-(1-g)$ for all $x\in \tilde C$. The line bundles $\overline {\wedge^2\mathcal{F}''}$ and $\mathcal{G}$ are isomorphic outside the divisor $\mathcal{X}_{_0}\subset \mathcal{X}_{_{\tilde C}}$. Therefore, $\overline {\wedge^2\mathcal{F}''}\simeq \mathcal{G}\otimes \mathcal{O}_{_{\mathcal{X}_{_{\tilde C}}}}(n\mathcal{X}_{_0})$ for some integer $n$. Therefore $\overline {\wedge^2\mathcal{F}''}|_{_{\mathcal{X}_{_0}-p}}\simeq G_{_1}|_{_{\mathcal{X}_{_0}-p}}$. Let $L=\overline {\wedge^2\mathcal{F}''}|_{_{\mathcal{X}_{_0}}}$. Since $deg(L)=deg(G_{_1})$ we have $L|_{_{X_{_1}}}=G_{_1}|_{_{X_{_1}}}(n_{_1}p)$ and $L|_{_{X_{_2}}}=G_{_1}|_{_{X_{_2}}}(n_{_2}p)$ for some integers $n_{_1}$, $n_{_2}$ with $n_{_1}+n_{_2}=0$.   %Then there exists completef \ref Propositio  dvr $\tilde{R}$ dominant over $R$ and morphisms $q_{_i}:\tilde{C}:=Spec \tilde R\to \overline{\Gamma}$ such that $q_{_i}(\tilde{C}-0)\subset \Gamma$ and $q_{_i}(0)=G_{_i}$ for $i=1,2$. Let 
\end{proof}
\end{proposition}
\begin{remark}\label{reldet1}
By similar arguments as in the proof of Proposition \ref{tran} we can show that $Det$ is a smooth morphism. Fix a section $\sigma:C\to \mathcal{J}$ such that $\sigma(0)=\xi$. This corresponds to a line bundle $\mathcal{L}$ over $\mathcal{X}$ such that $\mathcal{L}|_{_{\mathcal{X}_{_0}}}=\xi$. Let us denote $Det^{-1}(\sigma(C))$ by $\mathcal{M}_{_{\mathcal{L}}}$. Since both the varieties $\mathcal{M}$ and $\mathcal{J}$ are smooth we conclude that $\mathcal{M}_{_{\mathcal{L}}}$ is smooth over $\C$.
\end{remark}
%\addbibresource{refer1.bib}
\bibliographystyle{amsplain}
\bibliography{refer1}

\end{document}